\documentclass[11pt]{article}
\usepackage{latexsym}
\usepackage{amsmath,amsthm,amsfonts,amssymb}
\usepackage[utf8]{inputenc}
\usepackage{url}
\usepackage{enumerate}
\usepackage{mdwlist} 
\usepackage{graphicx}
\usepackage{xcolor}
\usepackage[color]{showkeys}
\usepackage[a4paper,textwidth=15cm,textheight=25cm]{geometry}
\usepackage{xspace}
\usepackage{export}
\usepackage{hyperref}
\hypersetup{
    colorlinks=true,
    linkcolor=blue,
    filecolor=blue,      
   urlcolor=blue,
 }

\newtheorem{thm}{Theorem}[section]

\newtheorem*{thm*}{Theorem}

\newtheorem{cor}[thm]{Corollary}
\newtheorem*{cor*}{Corollary}

\newtheorem{prop}[thm]{Proposition} 
\newtheorem*{prop*}{Proposition} 
\newtheorem*{properties*}{Properties} 
 
\newtheorem{lem}[thm]{Lemma} 
\newtheorem*{lem*}{Lemma}

\newtheorem*{claim*}{Claim} 
 
\newtheorem*{fact*}{Fact}

\newtheorem*{fait*}{Fait}

\newtheorem*{qst*}{Question} 
\newtheorem{qst}[thm]{Question} 
 
\newtheorem*{pb*}{Problem} 
\newtheorem{pb}[thm]{Problem} 
 
\newtheorem*{conj*}{Conjecture}

\newtheorem*{exo*}{Exercise}

\theoremstyle{remark}
\newtheorem{dfn}[thm]{Definition} 
\newtheorem*{dfn*}{Definition}

\newtheorem*{algo*}{Algorithm} 
\newtheorem*{rem*}{Remark}
\newtheorem{rem}[thm]{Remark}
\newtheorem*{example*}{Example}
\newtheorem{example}[thm]{Example}

\newcounter{numEnonceTmpInterne}
\newenvironment{enonce*}[1]{\theoremstyle{plain}\stepcounter{numEnonceTmpInterne}%
\def\a{enoncetmp\alph{numEnonceTmpInterne}}%
\newtheorem*{\a}{#1}\begin{\a}}{\end{\a}}

\makeatletter
\edef\@tempa#1#2{\def#1{\mathaccent\string"\noexpand\accentclass@#2 }}
\@tempa\rond{017}
\makeatother

\newcommand{\es}{\emptyset}
\renewcommand{\phi}{\varphi} 
\newcommand{\m} {^{-1}} 
\newcommand{\eps} {\varepsilon}

\newcommand {\ra} {\rightarrow}

\newcommand{\actson}{\curvearrowright}

\newcommand{\ul}[1]{\underline{#1}}

\newcommand{\dunion}{\sqcup}

\newcommand{\ie} {i.~e.\ }

\newcommand {\cala} {{\mathcal {A}}}

\newcommand {\calf} {{\mathcal {F}}}   
   
\newcommand {\calh} {{\mathcal {H}}}

\newcommand {\bbC} {{\mathbb {C}}}

\newcommand {\bbF} {{\mathbb {F}}}   
\newcommand {\bbG} {{\mathbb {G}}}

\newcommand {\bbN} {{\mathbb {N}}}

\newcommand {\bbR} {{\mathbb {R}}}

\newcommand {\bbZ} {{\mathbb {Z}}}   

\newcommand{\N}{{\mathbb{N}}}
\newcommand{\Z}{{\mathbb{Z}}}

\newcommand{\F}{{\mathbb{F}}}

 \usepackage{bbm}

\newcommand{\grp}[1]{\langle #1 \rangle}

\newcommand{\Out} {{\mathrm{Out}}}
\newcommand{\Hom} {{\mathrm{Hom}}}
\newcommand{\Aut} {{\mathrm{Aut}}}

\newcommand{\id} {\mathrm{id}}
\newcommand{\ad} {{\mathrm{ad}}}

\newcommand{\Mx}[1]{\begin{matrix}#1\end{matrix}}

\newcommand{\bMx}[1]{\begin{bmatrix}#1\end{bmatrix}}

\newcommand{\inv}[1]{\check #1}
\newcommand{\Lip}{{\mathrm{Lip}}}

\usepackage{ stmaryrd }

\renewcommand{\epsilon}{\varepsilon}
\newcommand{\HOM}[1][L_V]{\mathrm{Hom}_\Gamma(#1,\Gamma)}

\newcommand{\fois}{\circledast}

\newcommand{\conn}{connected}

\title{Algebraic groups over free and hyperbolic groups}
\author{Vincent Guirardel and Chloé Perin} 

\begin{document}
\maketitle 

\begin{abstract} We define an algebraic group over a group $G$ to be a variety - that is, a subset of $G^d$ defined by equations over $G$ - endowed with a group law whose coordinates can be expressed as word maps. In the case where $G$ is a torsion-free hyperbolic group and the underlying variety is irreducible we give a complete description of all algebraic groups.
\end{abstract}

\section{Introduction}

In classical algebraic geometry, an algebraic group over a field $k$ is an algebraic variety endowed with a group law which can be expressed by polynomials in the coordinates. We ask what happens when we replace the field $k$ by a group $\Gamma$, and polynomials by word maps with constants in $\Gamma$.

More precisely, given a group $\Gamma$, an equation over $\Gamma$ in the variables
$x_1, \ldots, x_d$ is an expression of the form $w(x_1, \ldots, x_d)=1$, 
where $w$
is a word in the variables $x_1^\pm,\dots,x_d^\pm$ with constants in $\Gamma$,
\ie an element of the free product $\Gamma * \grp{x_1, \ldots, x_d}$. 
Given a set $\Sigma$ of such words, the variety $V$ defined by $\Sigma$ is the subset of $\Gamma^d$ consisting of the $d$-tuples $(v_1,\dots, v_d)\in \Gamma^d$ 
satisfying all the equations 
$w(v_1,\dots,v_d)=1$ for $w\in\Sigma$. 
An algebraic map $F:V\ra V'$ between two algebraic varieties $V\subset \Gamma^d$ and $V'\subset \Gamma^{d'}$ is a map such that there exist words $w_1,\dots, w_{d'}\in \Gamma*\grp{x_1,\dots,x_d}$
such that for all $v=(v_1,\dots,v_d)\in V$, 
$F(v)=(w_1(v),\dots, w_{d'}(v))$.
An isomorphism of algebraic varieties $F:V\ra V'$ is a bijection such that $F$ and $F\m$ are algebraic maps.
When $\Gamma$ is a non-elementary torsion-free hyperbolic group,
the set of algebraic varieties is the collection of closed sets for 
a topology, called the Zariski topology (the hypothesis on $\Gamma$ ensures that 
this is stable under finite unions, see \cite[\S3.1]{BMR_algebraicI}).

We define an algebraic group over $\Gamma$ as a group $V$ which is a variety 
and whose multiplication law  $\mu: V \times V \to V$ is an algebraic map.
It would be natural to ask additionally that the map $V\ra V$ sending and element to its inverse 
is also an algebraic map, but we actually don't require this as our results hold without
this assumption (and in fact imply that this extra condition automatically holds, see Remark \ref{rem_inverse}).

\begin{example} Let $\Gamma$ be any group. Then 
$\Gamma$ itself is an algebraic group
for its standard multiplication law $\mu(u,v)=uv$, 
which is simply the algebraic map $\Gamma^2\ra \Gamma$ associated to the word $xy\in \Gamma*\grp{x,y}$.
We note that the inverse $v\mapsto v\m$ is also an algebraic map.
\end{example}

\begin{example} 
If $a\in \Gamma$ is any element, then its centralizer $V=Z_\Gamma(a)$
is an algebraic group with underlying variety defined by the equation $xa x\m a\m=1$,
and endowed with the standard multiplication law.
Similarly, the centralizer of any subset of $\Gamma$ is an algebraic group.
Now if $\Gamma$ is a torsion-free hyperbolic group, 
any centralizer is either the trivial group, the whole group $\Gamma$, or a maximal cyclic subgroup $\grp{c}=Z(c)$.
\end{example}

\begin{example}\label{ex_general}
The direct product of finitely many algebraic groups over $\Gamma$ is again an algebraic group. 
Thus, given $a_1, \ldots, a_l\in \Gamma\setminus\{1\}$ and $n_1,\dots,n_l,r\in \bbN$, the product
$\Gamma^r\times Z(a_1)^{n_1}\times \dots \times Z(a_l)^{n_l}$
is also an algebraic group.
\end{example}

\begin{example}
If $a$ is conjugate to $a'=bab\m$ then $V=Z(a)$ is isomorphic to $V'=Z(a')$ as an algebraic group, the isomorphism $v\mapsto bvb\m$ and its inverse being algebraic maps.
For this reason, if $\Gamma$ is a torsion-free hyperbolic group,
any group as in Example \ref{ex_general} is isomorphic, as an algebraic group, to 
a group of the form 
$\Gamma^r\times \grp{c_1}^{n_1}\times \dots \times \grp{c_l}^{n_l}$
where $\grp{c_1},\dots,\grp{c_l}$ are non-conjugate maximal cyclic groups.
\end{example}

\begin{example}
If $V$ is an algebraic group and $F:V\ra V'$ is an isomorphism of algebraic varieties,
one can transport the multiplication law by $F$. 
For example, conjugating the standard multiplication law on $V=\Gamma^2$ by the algebraic map 
$F:\bMx{u_1\\u_2}\mapsto\bMx{u_1u_2\\ u_2u_1u_2}$,
one gets a seemingly complicated multiplication law $\fois$ on $V'=V=\Gamma^2$ given by
$$\bMx{u_1\\u_2}\fois \bMx{v_1\\v_2} = 
\bMx{u_1^2 u_2\m v_1^2 v_2\m u_2u_1\m v_2v_1\m\\
u_2u_1\m v_2v_1\m u_1^2 u_2\m v_1^2 v_2\m u_2u_1\m v_2v_1\m
}$$
but this is just an algebraic group isomorphic to $\Gamma^2$ with its standard multiplication law.
Here we say that two algebraic groups $(V,\mu),(V',\mu')$ are isomorphic (as algebraic groups over $\Gamma$)
if there is an isomorphism of algebraic varieties $F:V\ra V'$ which is a group isomorphism.
\end{example}

In this paper we work over a fixed torsion-free hyperbolic group $\Gamma$.
A variety $V$ is called \emph{irreducible} if it is not the union of two proper subvarieties.
For an algebraic group, it is irreducible if and only if it is connected for the Zariski topology (see Section \ref{sec_irreducible}). 
We thus say that an algebraic group $V$ is \emph{\conn} if the underlying variety is connected, or equivalently irreducible.
Focusing on \conn\ algebraic groups is not a big restriction since any algebraic group has a finite index subgroup which is \conn\ (this is a consequence of equational Noetherianity due to Sela \cite{Sela_diophantine7}, see Section \ref{sec_irreducible}).
All the examples of algebraic groups given above are \conn.

Our main theorem gives a classification of all \conn\ algebraic groups over $\Gamma$,
and says that all of them come from the examples above.

\begin{thm} \label{thm_main_result} Let $\Gamma$ be a torsion-free hyperbolic group.
Let $V$ be a \conn\ algebraic group over $\Gamma$.

Then there exist non-conjugate maximal cyclic subgroups $\grp{c_1}, \ldots, \grp{c_l}$ of  $\Gamma$, and integers $r \in \N$ and $n_1, \ldots, n_l \in \bbN\setminus\{0\}$ such that 
$V$ is isomorphic, as an algebraic group, to
$$ \Gamma^r \times \grp{c_1}^{n_1} \times \ldots \times \grp{c_l}^{n_l}$$
where the multiplication law is the standard coordinate-wise product.

Moreover, this gives a complete classification of \conn\ algebraic groups up to isomorphism:
the conjugacy classes of the groups $\grp{c_1},\dots,\grp{c_l}$ and the exponents $n_i$ and $r$ are uniquely determined
(up to permutation) by the isomorphism class of $(V, \mu)$ as an algebraic group.
\end{thm}

\begin{rem} \label{rem_inverse} The theorem implies in particular that for any algebraic group $V$ over $\Gamma$, 
the map $V \to V$ sending an element to its inverse is algebraic: in $\Gamma^r\times \grp{c_1}\times\dots \times \grp{c_n}$,
this map is given by $(x_1,\dots,x_r,y_1,\dots y_n)\mapsto
(x_1\m,\dots,x_r\m,y_1\m,\dots y_n\m)$.
\end{rem}

\begin{rem}
As mentioned above, if $V$ is an algebraic group which is not assumed to be connected, 
then the \conn\ component of $V$ containing the neutral element 
is a finite index subgroup which is itself a connected algebraic group.
However, with the definition given above, 
this remark is not enough to deduce a description of all algebraic groups from our result on connected ones. Indeed,
it is not clear even which finite groups can be realised as algebraic groups:
every finite set is an algebraic set but its is not clear whether a given group law can
be expressed as a word map.

With this in mind, it would be more natural to enlarge the category of algebraic groups by allowing the group law to be
given by word maps only locally. For connected algebraic groups, this definition is equivalent
to the previous one. However, it has the advantage of being stable under taking supergroups of finite index: with the new definition,
if $G$ has a finite index subgroup which is isomorphic to a connected algebraic group,
then $G$ itself is isomorphic to an algebraic group.
\end{rem}

\subsection*{Other notions and questions}
One can consider more general notions.
A group $(V,\mu)$ is \emph{definable} over a structure $\Gamma$ (such as a field, or a group) 
if $V$ is a subset of $\Gamma^d$ definable by a first order formula, 
and $\mu$ (or more precisely the graph of $\mu$) is also definable in this way.
Even more generally, a group $(V,\mu)$ is \emph{interpretable} over $\Gamma$
if $V=\tilde V/\sim$ where $\tilde V$ is a definable subset and $\sim$
is a definable equivalence relation, and similarly for the graph of $\mu$.

Over an algebraically closed field, the Weil-Hrushovski theorem states
that any interpretable group is an algebraic group 
(see \cite{Poizat_stable} Theorem 4.13, with (3) from \S 4.1), 
where the notion of algebraic group here should not be restricted to affine algebraic groups but should also include abelian varieties for example.

Let us now consider groups interpretable over a base group $\Gamma$. 
If
$\Gamma$ is a group such as $SL_n(\bbC)$
or the Heisenberg group $\Gamma=H_3(\bbC)$ (or more generally an algebraic group over $\bbC$, thought of as an abstract group), 
then any interpretable group over $\Gamma$ will also be interpretable over $\bbC$. Thus
the theorem by Weil-Hrushovski shows that all interpretable groups
over $\Gamma$ are actually isomorphic to algebraic groups over $\bbC$.
On the other hand, starting from the Heisenberg group $\Gamma=H_3(R)$
over a ring $R$, then the ring $(R,+,.)$ is interpretable over the group $\Gamma$ \cite{Malcev_rings_and_groups}, 
so conversely any algebraic group over $\bbC$ is interpretable over $H_3(\bbC)$.
But taking $R=\bbZ$, one gets that arithmetic is interpretable over $H_3(\bbZ)$
and it easily follows that every finitely presented group 
is interpretable over $H_3(\bbZ)$.

Returning to the case where $\Gamma$ is a torsion-free hyperbolic group (and even a free group), we ask:
\begin{pb}
        Let $\Gamma$ be a torsion-free hyperbolic group. Classify the definable and interpretable groups over $\Gamma$.
\end{pb}

We note that Byron and Sklinos proved, 
relying on Sela's work on the first order theory of free groups, that no infinite field is
interpretable over the free group  \cite{ByronSklinos_fields,Sklinos_fields,Sela_diophantine6,Sela_diophantine9}.

\subsection*{Ingredients of the proof.}
Our proof of Theorem \ref{thm_main_result} uses
basic tools from algebraic geometry over groups \cite{BMR_algebraicI}.
The central notion is that of the \emph{group of functions} $L_V$ over a variety $V$ (also called \emph{coordinate group} in the literature). 
For instance, in the case where $\Gamma$ is a free group, coordinate groups $L_V$ are precisely finitely generated residually free groups. When $V$ is irreducible (see below), 
$L_V$ is actually a \emph{limit group}, a class of well behaved groups 
which has been widely studied  \cite{Sela_diophantine1,KhMy_irreducible2,Sela_diophantine7,GrWi_structure}.
The group $L_V$ contains a preferred copy of $\Gamma$ corresponding to the constant functions
(except for the degenerate case $V=\es$).
Given a point $v$ in $V$, the evaluation homomorphism $L_V\ra \Gamma$ 
is defined by evaluating a function
$f\in L_V$ at the point $v$.
This gives a bijective correspondence which identifies $V$
with the set $\HOM$ of all homomorphisms $L_V\to \Gamma$
that restrict to the identity on the preferred copy of $\Gamma$ in $L_V$.

Of fundamental importance will be the functoriality of the correspondence between $V$ and $L_V$: a morphism of algebraic varieties $\alpha: V \to W$ corresponds to a group morphism $\alpha^*:L_W \to L_V$, and conversely. 
In particular, the group of automorphisms of $V$ (as an algebraic variety)
is isomorphic to $\Aut_\Gamma(L_V)$, the group of automorphisms of $L_V$
that are the identity on $\Gamma$.
For example, if $V$ is homogeneous under the action of its automorphism group,
then $\Aut_\Gamma(L_V)$ acts transitively on $\HOM$.
Using tools such as Grushko decompositions and JSJ decompositions, which help describe automorphisms of $L_V$,
we get strong
restrictions on $L_V$ and thus on $V$.

However this is not enough to complete the description of algebraic groups, since 
there are varieties $V$ that are homogeneous but do not correspond
to algebraic groups over $\Gamma$. 
In Section \ref{sec_ex_homogeneous_variety} for instance, we describe a variety $V$ over $\Gamma$ a free group, 
on which there is a free transitive action
of the mapping class group of a surface, but it follows from our classification that $V$ does not admit a structure of algebraic group.

If $V$ is an algebraic group, its action on itself yields a representation
of $V$ into the group $\Aut_\Gamma(L_V)$.
The fact that this representation comes from an algebraic action of an algebraic
group gives strong restrictions on the image of $V$ in $\Aut_\Gamma(L_V)$.
For example, in algebraic geometry over a field $k$, if an algebraic group acts
algebraically on the affine space $k^n$, then any orbit for its action on the ring of functions $k[X_1,\dots X_n]$
has bounded degree.
The analogous statement here says that the action of $V$ on the group of functions $L_V$
has bounded stretch:
there exists a constant $C$ such that all elements of $V$ act
on $L_V$ with Lipschitz constant at most $C$
(see Lemma \ref{lem_alg_action}) where the length of elements of $L_V$ 
is measured in a word metric
where all elements of $\Gamma\setminus\{1\}$ have length 1.

Because of the lack of local finiteness of the metric involved, 
this does not imply finiteness of the image of $V$.
For instance, if $L_V=\Gamma*\grp{X_1,\dots,X_n}$,
and given two tuples $a=(a_1,\dots,a_n),b=(b_1,\dots,b_n)\in \Gamma$, 
the automorphism $\tau_{a,b}:X_i\mapsto a_iX_ib_i$
is an automorphism whose powers have bounded Lipschitz constants,
and which usually has infinite order.
Still, this gives strong restrictions on the representation of $V$, and for instance, staying in the case where $L_V=\Gamma*\grp{X_1,\dots,X_n}$,
it prevents its image from containing complicated automorphisms such as 
fully irreducible automorphisms 
(in the sense of free products relative to $\Gamma$).
Applying a result by Handel and Mosher \cite{HaMo_subgroup} 
and generalized by Horbez and the first author \cite{GH_boundaries},
one deduces that there is a triangular structure invariant under the image of $V$ in $\Aut_\Gamma(L_V)$ (the image of $V$ looks like a unipotent group).

The endgame is then algebraic: using the fact that the action of $V$
on $V$ is triangular, one can then express the associativity of the group law to deduce that our algebraic group is isomorphic to a standard one.

\subsection*{Organization of the paper}

The paper is structured as follows: in Section \ref{sec_alg_geom_tools}, we survey tools of algebraic geometry over groups and the function group
$L_V$ associated to a variety $V$.

In Section \ref{sec_gamma_limit}, we recall some results about $\Gamma$-limit groups, in particular we define 
the JSJ decomposition for such groups and give some of their properties. 

In Section \ref{sec_basic_alg_groups}, we consider algebraic groups over $\Gamma$. We prove a useful formula expressing the group law under the identification of $V$ with $\Hom_{\Gamma}(L_V)$, and express homogeneity of the underlying variety under this same identification.

We use this in Section \ref{sec_gp_functions_homogeneous} to deduce a number of constraints that the JSJ decomposition of the function group $L_V$ of a homogeneous irreducible variety must satisfy.

In Section \ref{sec_trimming}, 
we construct a canonical algebraic subgroup $(W, \mu|_{W \times W})$ of $(V, \mu)$ 
associated to the abelian factors in $V$.
We deduce that the JSJ decomposition of $L_W$ and of $L_V$ 
have no non-abelian rigid vertices other than the one containing the constants group $\Gamma$. 

In Section \ref{sec_bounded_stretch}, we introduce the bounded stretch argument
and use it to show that the JSJ decomposition of $L_V$ cannot contain any surface type vertices either. 

This enables us in Section \ref{sec_underlying_variety} to give a complete description of the various possibilities for the variety $V$, via a description of the JSJ decomposition of $L_V$. 

Finally, in Section \ref{sec_mult_table},
we use the bounded stretch argument to find a triangular structure
of the image of $V$ and $V^2$ in $\Aut_\Gamma(L_V)$,
and use it to find a change of coordinates 
under which the group law becomes the standard coordinate-wise product, 
thus completing the proof of Theorem \ref{thm_main_result}.

\subsection*{Acknowledgement}
The first author was partially supported by  European Research Council (ERC GOAT 101053021), by the french ANR grant ANR-22-CE40-0004 GOFR, 
and benefited from the support of the French government ”Investissements d’Avenir” program integrated to France 2030 (ANR-11-LABX-0020-01).
The second author was supported by the ISRAEL SCIENCE FOUNDATION (grant No. 2176/20) and CNRS.

\section{Algebraic geometric tools over groups}\label{sec_alg_geom_tools}

In all this paper, we fix $\Gamma$ a torsion-free hyperbolic group
which we assume to be non-elementary 
(i.e.~containing a non-abelian free group).

\subsection{$\Gamma$-groups.} Following \cite{BMR_algebraicI}, 
a \emph{$\Gamma$-group} $G$ is a group $G$ together
with a specified monomorphism $\iota:\Gamma\ra G$. 
A 
\emph{$\Gamma$-homomorphism} $\phi:G\ra G'$ is a homomorphism that commutes with the specified embeddings.
We will usually keep $\iota$ implicit, and identify $\Gamma$ with its image;
thus, a $\Gamma$-homomorphism is a homomorphism that restricts to the identity on $\Gamma$.
We denote by $\Hom_\Gamma(G,G')$ the set of $\Gamma$-homomorphisms.

For instance $\Gamma*\grp{x_1,\dots,x_d}$ (where $\grp{x_1,\dots,x_d}$ is the free group on the $x_i$'s) is a $\Gamma$-group 
whose elements are called $\Gamma$-words: 
they can be viewed as
words in the variables $x_1^\pm,\dots,x_d^\pm$ with coefficients in $\Gamma$.

A subset $S$ of a $\Gamma$-group $G$ is a \emph{$\Gamma$-generating set} if
$\Gamma\cup S$ generates $G$ as a group.

\begin{dfn} A $\Gamma$-group $G$ is residually $\Gamma$ if for any $g \in G \setminus \{1\}$ there exists a $\Gamma$-homomorphism $h: G \to \Gamma$ such that $h(g) \neq 1$.
\end{dfn}

See Definition \ref{dfn_fully} for the stronger notion of being fully residually $\Gamma$.

\begin{dfn}
Given a $\Gamma$-group $G$, its \emph{largest residually $\Gamma$ quotient} is the group
$$G/K \quad \text{where } K=\bigcap_{h\in \HOM[G]}\ker h$$ 
\end{dfn}

Clearly, $G/K$ is residually $\Gamma$ and if $G\to G'$ is a $\Gamma$-homomorphism to
a $\Gamma$-group $G'$ which is residually $\Gamma$,
it factors through $G/K$.

If $G=\grp{\Gamma,x_1,\dots,x_d|\Sigma}$ for a collection $\Sigma$ of $\Gamma$-words, then $K$ is called the \emph{radical} of $\Sigma$
in \cite{BMR_algebraicI}.

\subsection{Variety over $\Gamma$.} A system of equations over $\Gamma$ in the variables $x_1, \ldots, x_d$ is a set of words 
$$\Sigma=\{ w_i(x_1, \ldots, x_d) \mid i \in I\}\subset \Gamma * \langle x_1, \ldots, x_d \rangle.$$
The variety $V_{\Sigma} \subseteq \Gamma^d$ it defines is
$$V_{\Sigma} = \{ (v_1, \ldots, v_d) \in \Gamma^d \mid w_i(v_1, \ldots, v_d)=1 \textrm{ for each } i \in I\}.$$

Following Lemma 6(6) p.~49  and Proposition 1 p.~31 in \cite{BMR_algebraicI},
since $\Gamma$ is CSA and non abelian, the collection of varieties is stable under finite union and is the collection of closed sets of a topology on $\Gamma^d$, called the Zariski topology (see Theorem 3 p.~50 in \cite{BMR_algebraicI}).

\subsection{Group of functions.} Let $V$ be a subset of $\Gamma^d$. A word map on $V$ is a map $f:V \to \Gamma$ for which there exists a word
$w(x_1, \ldots, x_d) \in \Gamma * \langle x_1, \ldots, x_d \rangle$ 
such that for
all $(v_1,\dots,v_d)\in V$,
$f(v_1, \ldots, v_d) = w(v_1, \ldots, v_d)$.

For example, the projection on the $k$-th coordinate $V \to \Gamma, (v_1, \ldots, v_d) \mapsto v_k$ is a word map defined by the $\Gamma$-word $w(x_1, \ldots, x_d)=x_k$. We call this map the $k$-th coordinate function,
and we denote it by $X_k$. 

\begin{dfn}[Group of functions] 
The group of functions $L_V$ on $V$ is the set of word maps $f: V \to \Gamma$, endowed with multiplication in the range: given $f, f' \in L_V$ we define $ff': V \to \Gamma$ by setting $ff'(v)=f(v)f'(v)$.

\end{dfn} 

The group of functions on $V$ is also known as its \emph{coordinate group} \cite{BMR_algebraicI}.

All the varieties we consider will be non empty. In this case, $L_V$ is a $\Gamma$-group where $\Gamma$ embeds in $L_V$ as the set of constant maps.

By definition of a word map, there is a surjective map $\Gamma*\grp{x_1,\dots, x_d}\ra L_V$ sending a word $w$ to the corresponding word map, and sending the generator $x_k$ 
to the coordinate function $X_k$. 
Thus $(X_1,\dots,X_d)$ is a $\Gamma$-generating family of $L_V$.

\subsection{Equations of the variety and functions on $V$.}
If $V$ is a variety defined by $\Sigma=\{w_i(x_1,\dots,x_d) \mid i \in I\}$, then the map $\Gamma*\grp{x_1,\dots,x_d}\ra L_V$ vanishes on each $w_i$, hence
factors through the group $ L_{\Sigma}=\grp{\Gamma,x_1,\dots,x_d\mid w_i, i\in I}$.

Given a point $v=(v_1,\dots,v_d)\in V$, we denote the evaluation homomorphism by $h_v:\Gamma*\grp{x_1,\dots, x_d}\ra \Gamma$, it is defined by
$h_v(w)=w(v_1,\dots,v_d)$ for any word $w\in \Gamma*\grp{x_1,\dots, x_d}$. 
This is a $\Gamma$-homomorphism so $h_v\in \Hom_\Gamma(\Gamma*\grp{x_1,\dots, x_d},\Gamma)$.

Every evaluation homomorphism $h_v$ factors through $ L_{\Sigma}$
and through $L_V$. 
We still call the factored map $ L_\Sigma\ra \Gamma$ or $L_V\ra \Gamma$ an evaluation homomorphism and denote it by $h_v$.

\begin{lem} \label{lem_RF}
Every $h\in \Hom_\Gamma( L_\Sigma ,\Gamma)$
agrees with some evaluation homomorphism $h_v:L_{\Sigma} \to \Gamma$.
\end{lem}

\begin{proof}
Given $h\in \Hom_\Gamma( L_\Sigma ,\Gamma)$,
consider the point $v_h=(h(\bar x_1),\dots,h(\bar x_d))\in \Gamma^d$
where we denote by $\bar x_i\in L_\Sigma$ the image of $x_i$ 
under the morphism $\Gamma*\grp{x_1,\dots,x_d}\ra L_\Sigma$. 
The point $v_h$ lies in the variety $V$. 
Indeed, for each $w_i\in \Sigma$,
we have $w_i(h(\bar x_1),\dots,h(\bar x_d))=h(w_i(\bar x_1,\dots,\bar x_d))=1$  because $h$ is a $\Gamma$-morphism.
We check that $h=h_{v_h}$: 
given any word $w\in \Gamma*\grp{x_1,\dots,x_d}$ with image $\bar w\in  L_\Sigma$, we have
$$h_{v_h}(\bar w)=w(h(\bar x_1),\dots,h(\bar x_d)) = h( w(\bar x_1, \ldots, \bar x_d)) = h(\bar w).$$
This concludes the proof.
\end{proof}

Two words in $\Gamma*\grp{x_1,\dots,x_d}$ represent the same function
on $V$ if and only if they have the same image under all evaluations maps $h_v$ for $v\in V$.
Thus one can alternatively describe $L_V$ as the quotient  
$$L_V= L_\Sigma/\bigcap_{v\in V}\ker h_v.$$

By construction, $L_V$ is residually $\Gamma$: for every function $f\in L_V\setminus \{1\}$
there is a $\Gamma$-morphism $L_V\ra \Gamma$ which does not kill $f$, namely the evaluation map $h_v$
at a point $v\in V$ where $f$ does not vanish. 

As an immediate consequence of Lemma \ref{lem_RF} we get
\begin{cor}\label{cor_maxRF}
If $V$ the variety defined by the set of equations $\Sigma$, then
 $L_V$ is the largest 
 quotient of $ L_\Sigma$ that is residually $\Gamma$.
 
 In particular, the canonical map from $\Hom_{\Gamma}(L_{V}, \Gamma)$ to $\Hom_{\Gamma}(L_{\Sigma}, \Gamma)$ is a bijection.
\end{cor}

\begin{example}  Fix an element $c\in \Gamma\setminus\{1\}$ which is not a proper power so that the centralizer $Z(c)$ of $c$ is $\grp{c}$.
Let $V = Z(c)^d \subset \Gamma^d$
 be the variety defined by the system of equations $\Sigma=\{ [x_1, c],\dots, [x_d,c] \}$. 
Here $L_{\Sigma} = \grp{\Gamma,x_1,\dots,x_d\mid [x_1,c], \ldots ,[x_d, c]}$
and $$L_V \simeq \Gamma *_{\grp{c}} (\grp{c}\oplus \bbZ^d)$$ 
is its largest residually $\Gamma$ quotient (see for instance \cite[Theorem 1]{BMR_discriminating} to see it is indeed residually $\Gamma$).
\end{example}

\subsection{Correspondence $V \leftrightarrow \Hom_{\Gamma}(L_V, \Gamma)$.} 

Consider the maps
$$\left\{\Mx{V&\ra&\Hom_\Gamma(L_V,\Gamma)\\ v &\mapsto& h_v}\right.
\text{\quad and \quad}
\left\{\Mx{\Hom_\Gamma(L_V,\Gamma)&\ra& V\\ h &\mapsto& v_h=(h(X_1),\dots,h(X_d))}\right.
$$
The first associates to a point of $V$ its evaluation map, the second is the map $h \mapsto v_h$ that appears in the proof of the previous lemma, up to the bijection $\Hom_{\Gamma}(L_{V}, \Gamma) \rightarrow \Hom_{\Gamma}(L_\Sigma, \Gamma)$ mentioned in Corollary \ref{cor_maxRF}.

We claim that these maps are inverses of each other.
 Indeed, we already proved in Lemma \ref{lem_RF} that $h_{v_h}=h$ for all $h\in \Hom_\Gamma( L_{\Sigma},\Gamma)$. This is true on $\Hom_{\Gamma}(L_V, \Gamma)$ via the canonical bijection between $\Hom_{\Gamma}( L_{\Sigma}, \Gamma)$ and $\Hom_{\Gamma}(L_V, \Gamma)$. 
 For the other direction, we consider $v \in V$ and $h_v:L_V\ra \Gamma$ the corresponding evaluation morphism.  
 We have $$v_{h_v}=(h_v(X_1),\dots,h_v(X_d))=(X_1(v),\dots,X_d(v))=v.$$

\subsection{Algebraic maps. } Let $V \subseteq \Gamma^d ,V' \subseteq \Gamma^{d'}$ be algebraic varieties over $\Gamma$. An algebraic map from $V$ to $V'$ is a map 
$F: V \to V'$ 
such that there exist word maps $f_1,\dots,f_{d'}:V\ra \Gamma$ such that 
$F(v)= (f_1(v), \ldots, f_{d'}(v))$ for all $v\in V$.
An algebraic automorphism  $F:V\ra V$ is a bijective algebraic map whose inverse is an algebraic map.
  
  Recalling that elements of $L_V$ (respectively $L_{V'}$) are word maps $V \to \Gamma$ (respectively $V' \to \Gamma$) 
 any algebraic map $F:V\ra V'$ induces a $\Gamma$-homomorphism
  $$ F^*:\left\{\Mx{ L_{V'} &\to& L_V \\ f &\mapsto &f \circ F.}\right.$$
  Note that given a variety $V''$ and an algebraic map $G: V' \to V''$, we have $(G \circ F)^*=F^* \circ G^*$.
  In other words, there is a contravariant functor from the category of algebraic varieties over $\Gamma$ to the category of $\Gamma$-groups that are residually $\Gamma$.

  Conversely, given a $\Gamma$-homomorphism $\theta: L_{V'} \to L_V$ 
  (with $\Gamma$-generating families $(X_1,\dots,X_d)$ and $(X'_1,\dots,X'_{d'})$ of $L_V$ and $L_{V'}$ respectively)
  there is an induced 
  map $\theta_{*}: \Hom_{\Gamma}(L_V, \Gamma) \to \Hom_{\Gamma}(L_{V'}, \Gamma)$ given by $h \mapsto h \circ \theta$. Via the identification between $V$ and $\Hom_{\Gamma}(L_V, \Gamma)$ (respectively between $V'$ and $\Hom_{\Gamma}(L_{V'}, \Gamma)$)
  induced by the $\Gamma$-generating families,
 this can be seen as an algebraic map $V \to V'$:
  the image of $X'_1,\dots, X'_{d'}\in L_{V'}$ under $\theta$ can be written 
  using the $\Gamma$-generating set 
  $\{X_1,\dots,X_{d}\}$ of $L_{V}$, say $\theta(X'_i)=u_i(X_1,\dots,X_{d})$
  for some $\Gamma$-words $u_i\in \Gamma*\grp{x_1,\dots,x_{d}}$,
  and the algebraic map $V\ra V'$ is defined by 
  $$(v_1,\dots, v_d)\in V\mapsto (u_1(v_1,\dots,v_d),\dots, u_{d'}(v_1,\dots,v_d)).$$

  In particular, there is an isomorphism between
  the group $\Aut(V)$ of algebraic automorphisms of the variety $V$, and the group 
  $\Aut_{\Gamma}(L_V)$ of $\Gamma$-automorphisms of the $\Gamma$-group $L_V$.
  
 \subsection{Noetherianity and irreducible varieties.} Sela proved the following
  \begin{thm}[\cite{Sela_diophantine7}] \label{thm_eq_Noeth}  
Torsion-free hyperbolic groups are equationally Noetherian: for any system of equations $\Sigma\subset \Gamma*\grp{x_1,\dots,x_n}$,
there exists a finite subsystem $\Sigma'\subset \Sigma$ 
such that $V_{\Sigma}=V_{\Sigma'}$.
  \end{thm}
  
Equivalently, any descending chain of varieties $V_1 \supset V_2 \supset V_3 \ldots$ is stationary. 
  
  \begin{dfn}
  A non-empty variety is irreducible if it cannot be written as a finite union of proper subvarieties.
  \end{dfn}
  
  As a straightforward corollary of Theorem \ref{thm_eq_Noeth} we get that any variety $V$ can be written as a finite union $V_1 \cup \ldots \cup V_s$ 
  of irreducible varieties. 
Additionally, there is a unique such decomposition for which $s$ is minimal 
(equivalently, no inclusion $V_i\subset V_j$ holds for $i\neq j$): in this case, $V_1,\dots,V_s$ are called the \emph{irreducible components} of $V$.

\begin{dfn}
Let $G,G'$ be two $\Gamma$-groups. A collection $\calh$ of $\Gamma$-homomorphisms 
$G\ra G'$ is \emph{discriminating} if
for every finite subset $F\subset G$
there exists $h\in \calh$ such that $h$ is injective in restriction to $F$.
\end{dfn}

\begin{rem}\label{rem_discriminating}
For a sequence of $\Gamma$-homomorphisms, 
we will use the following variant of this definition: 
a sequence of $\Gamma$-homomorphisms $h_i:G\ra G'$ ($i\in \bbN$) is discriminating if
for every finite subset $F\subset G$, for all $i$ large enough, $h_i$
is injective in restriction to $F$. 
Note that for $G$ countable, a collection $\calh$ of $\Gamma$-homomorphisms is discriminating if and only if it contains a discriminating sequence.
\end{rem}

Equivalently, $\calh$ is discriminating if for any finite collection of
non-trivial elements $g_1,\dots,g_n\in G\setminus\{1\}$,
there exists $h\in \calh$ such that $h(g_1),\dots,h(g_n)$ are all non-trivial.

\begin{dfn}\label{dfn_fully}
A $\Gamma$-group $G$ is \emph{fully residually $\Gamma$} if
it admits a discriminating family of $\Gamma$-homomorphisms to $\Gamma$
(equivalently $\Hom_\Gamma(G,\Gamma)$ is discriminating).
\end{dfn}
  
  Sela proved that the class of finitely generated $\Gamma$-groups which are fully residually $\Gamma$ coincide with that of $\Gamma$-limit groups \cite{Sela_diophantine7}
  (actually, in \cite{Sela_diophantine7} \S1, Sela works with equations which do not involve constants in $\Gamma$, so the $\Gamma$-limit groups there don't come with a preferred copy $\Gamma$).

The following lemma 
is a slight variation on \cite[Theorem D2]{BMR_algebraicI}.
\begin{lem}\label{lem_discriminating_V0}
The variety $V$ is irreducible if and only if $L_V$ is fully residually  $\Gamma$.

In this case, if $V_0\subset V$ is a dense subset (for the Zariski topology),
    then the family of evaluation homomorphisms $\{h_v | v\in V_0\}\subset \HOM$ is discriminating on $L_V$.
\end{lem}
  
 \begin{proof}
 We first prove the second assertion by contradiction: we assume that $V$ is irreducible,
 but that $\{h_v|v\in V_0\}$ is not discriminating.
   Then there exists $f_1,\dots,f_n\in L_V\setminus\{1\}$
    such that for each $v\in V_0$, there exists $i\leq n$ such that $h_v(f_i)=f_i(v)=1$.
    Denote by $V_i= f_i\m(\{1\})$ the subvariety of $V$ 
    defined by $f_i$.
    In other words, $V_0\subset \bigcup_{i} V_i$. 
    Since $V_0$ is dense and the $V_i$ are closed, we get $V=\bigcup_{i} V_i$
    so by irreducibility, $V=V_i$ for some $i\leq n$.
    This means that $f_i$ vanishes on $V$, \ie $f_i=1$ a contradiction.
 
 This proves the second assertion. The direct implication of the first assertion follows.
  
  Conversely, assume that $V=V_1\cup\dots\cup V_n$ is a union of proper subvarieties.
  For each $i\leq n$, consider a word map $f_i\in L_V\setminus\{1\}$ that vanishes on $V_i$.
  Then for any point $v\in V$, $v$ lies in $V_i$ for some $i\leq n$ and the evaluation homomorphism $h_v$ vanishes on $f_i$.
This shows that the collection of evaluation homomorphisms is not discriminating. Since by Lemma \ref{lem_RF} every $\Gamma$-morphism to $\Gamma$ is an evaluation homomorphism, this proves the lemma.
 \end{proof}

\subsection{Direct products}

Let $V_1,V_2$ be two non-empty algebraic varieties, 
and let $V=V_1\times V_2$ be the product variety.
We now describe the group of functions $L_V$ in terms of $L_{V_1}$ and $L_{V_2}$.

For $i\leq 2$, denote by $p_i:V\ra V_i$ the $i$-th projection.
Recall that the map $p_i^*:L_{V_i}\ra L_V$
sends a word map $f:V_i\ra \Gamma$ to $f\circ p_i$.
The two maps $p_1^*, p_2^*$ are injective, they agree on $\Gamma$,
and thus yield a $\Gamma$-homomorphism $\phi:L_{V_1}*_\Gamma L_{V_2}\ra L_V$. 

\begin{lem}\label{lem_direct_product}
    $L_V$ is the maximal residually $\Gamma$ quotient of $L_{V_1}*_\Gamma L_{V_2}$.
    
    More precisely, the map $L_{V_1}*_\Gamma L_{V_2}\ra L_V$ is surjective and its kernel is exactly
    the intersection of all kernels of $\Gamma$-homomorphism $L_{V_1}*_\Gamma L_{V_2}\ra \Gamma$.
\end{lem}

\begin{proof}
Let $\Sigma_1\subset \Gamma*\grp{x_1,\dots,x_{d}}$ (resp.\ $\Sigma_2\subset\Gamma*\grp{y_1,\dots,y_{d'}}$) 
be a system of equations defining $V_1$ (resp. $V_2$).
By Lemma \ref{lem_RF}, $L_{V_1}$ is the maximal residually-$\Gamma$ quotient of the group $$L_{\Sigma_1}=\grp{\Gamma,x_1,\dots,x_d|\Sigma_1}$$
(\ie its quotient by the intersection $K_1$ of the kernels 
of all $\Gamma$-homomorphisms  $L_{\Sigma_1}\to\Gamma$), and similarly for $L_{V_2}$.
Since $\Sigma_1\cup\Sigma_2\subset\Gamma*\grp{x_1,\dots,x_d,y_1,\dots,y_{d'}}$ 
is a system of equations defining $V$,
$L_V$ is the maximal residually $\Gamma$ quotient of 
$$L_{\Sigma_1}*_\Gamma L_{\Sigma_2}=\grp{\Gamma,x_1,\dots,x_d,y_1,\dots,y_{d'}|\Sigma_1\cup \Sigma_2}.$$
Clearly $K_1$ and $K_2$ die in $L_V$ so 
$L_V$ is also the maximal residually $\Gamma$ quotient of $L_{V_1}*_\Gamma L_{V_2}$.

\end{proof}

\begin{rem}
One can prove that if $V_1$ and $V_2$ are irreducible, then so is $V_1\times V_2$.
Indeed, the proof over fields given for instance in \cite[Exercise I.3.15]{Hartshorne} adapts directly.
\end{rem}

\begin{example}
\label{rem_example_prod}
The description given by the lemma is not very explicit. 
We will need a more precise description only in simple situations such as the following.
Start with $V_1=V_2=Z(c)$ where $c$ is not a proper power.
Then for $i\in\{1,2\}$, 
$$L_{V_i}=\Gamma*_{\grp{c}} (\grp{c}\oplus \grp{t_i})\simeq \Gamma*_{\grp{c}} \bbZ^2 $$ 
and $$L_{V_1\times V_2}=\Gamma*_{\grp{c}} (\grp{c}\oplus \grp{t_1}\oplus\grp{t_2}).$$ 
The map $L_{V_1}*_\Gamma L_{V_2}\ra L_{V_1\times V_2}$
is not injective: its kernel contains for example the commutator $[t_1,t_2]$.
\end{example}

\section{$\Gamma$-limit groups} \label{sec_gamma_limit}

\subsection{Description of the JSJ decomposition of $\Gamma$-limit groups}

In this section we follow the terminology of \cite{GL_JSJ}. 

Let $L$ be a $\Gamma$-limit group.
Given two collections $\cala,\calh$ of subgroups of $L$,
an $(\cala,\calh)$-tree is a tree with a minimal action of $L$
such that each edge stabilizer lies in $\cala$, and each group in $\calh$ is elliptic.
A subgroup $H\subset L$ is $(\cala,\calh)$-universally elliptic if it is elliptic in every $(\cala,\calh)$-tree.
An $(\cala,\calh)$-JSJ decomposition is an $(\cala,\calh)$-tree whose edge groups are 
$(\cala,\calh)$-universally ellipic, and maximal for domination (given two $(\cala,\calh)$-trees 
$T$, $T'$, we say that $T$ dominates $T'$ if vertex stabilizers of $T$ are elliptic in $T'$).

In what follows, we will take for $\cala$ the collection of abelian subgroups of $L$. 
Denoting by $\cala_{nc}\subset\cala$ the subcollection of all non-cyclic abelian groups,
we will take $\calh=\{\Gamma\}\cup\cala_{nc}$. Thus $(\cala,\calh)$-trees
are abelian splittings of $L$ relative to $\Gamma$ and to non-cyclic abelian groups.

If $L$ is freely indecomposable relative to $\Gamma$,
it has a canonical JSJ tree over abelian groups relative to $\Gamma$ and 
to non-cyclic abelian groups.
Indeed $L$ is torsion-free and CSA (because $\Gamma$ is) thus Theorem 9.5 of \cite{GL_JSJ} implies that the JSJ deformation space of $L$ relative to $\Gamma$ and to non-cyclic abelian subgroups exists. 
Cylinders of an $(\cala,\calh)$-tree $T$ are defined
as the union of edges whose stabilizers are contained in a common
maximal abelian subgroup.
The tree of cylinders of $T$ is defined by replacing
each cylinder of $T$ by the cone on its boundary vertices 
(see \cite[Section 7, example (1)]{GL_JSJ}).

The tree of cylinders of any tree $T$ in the JSJ deformation space of $L$ relative to $\Gamma$ is itself a JSJ tree that does not depend on $T$, and is invariant by $\Aut_{\Gamma}(L)$.
We denote by $\Lambda$ the corresponding graph of groups.
We call the decomposition $L=\pi_1(\Lambda)$ the \emph{canonical JSJ decomposition} of $L$. 
We will denote by $v_{\Gamma}$ the image in $\Lambda$ of the unique vertex $\tilde v_\Gamma$ of the JSJ tree
fixed by $\Gamma$. We denote by $G_{v_\Gamma}=G_{\tilde v_\Gamma}$ the corresponding vertex group (which contains $\Gamma$).

If we drop the assumption of free indecomposability,
we choose $L=L_0 * \ldots * L_r * \F_s$ a Grushko decomposition relative to $\Gamma$ 
 with $\Gamma \subset L_0$. 

As mentioned above, we can consider the canonical JSJ-decomposition $L_0=\pi_1(\Lambda_0)$ relative to $\Gamma$ and non-cyclic abelian groups. 
For each $i\geq 1$, $L_i$ has no preferred copy of $\Gamma$ (it contains no conjugate of $\Gamma$).
There is an $\Aut(L_i)$-invariant JSJ-decomposition $L_i=\pi_1(\Lambda_i)$ relative to non-cyclic abelian groups.

The following definition will be convenient.
\begin{dfn}[Standard JSJ decomposition] 
Gluing a wedge of $s$ edges with trivial edge group on $v_\Gamma$, and connecting $v_\Gamma$ to each graph of groups $\Lambda_i$ ($i>0$) using edges with trivial edge group (see Figure \ref{fig:jsj_std}), 
one obtains a splitting of $L$.
We call such a splitting a \emph{standard JSJ decomposition} of $L$.
\end{dfn} 

\begin{figure}
    \centering
    \includegraphics[width=.8\linewidth]{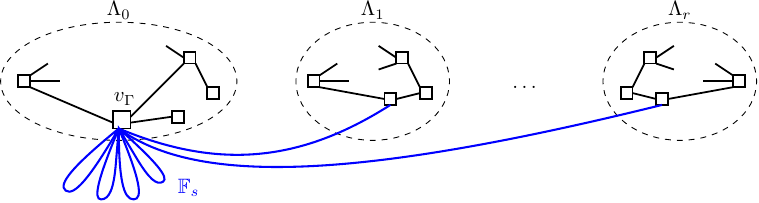}
    \caption{A typical standard JSJ decomposition. Blue edges have trivial edge group, black edges have non-trivial abelian edge group.}
    \label{fig:jsj_std}
\end{figure}

A standard JSJ decomposition of $L$ is indeed a JSJ decomposition of $L$ over abelian groups relative to $\Gamma$ and to non-cyclic abelian groups (but it is not invariant under 
$\Aut_\Gamma(L)$ in general).

Trees of cylinders can be characterized in the following way:
\begin{lem}\label{lem_cylindres}
    Let $G=\pi_1(\Lambda)$ be an abelian splitting of a CSA group $G$ without trivial edge groups.
    
    The tree corresponding to $\Lambda$ is its own tree of cylinders if and only if 
     it is bipartite between vertices with abelian and non-abelian stabilizer,
    and for every vertex $v$ with non-abelian stabilizer, 
    \begin{itemize}
        \item for every edge $e$ incident on $v$, the group $G_e$ is maximal abelian in $G_v$;
        \item for every pair of edges $e\neq e'$ incident on $v$,
    $G_e$ is not $G_v$-conjugate to $G_{e'}$.
    \end{itemize}
\end{lem}

\begin{proof} In this context, cylinders have maximal abelian stabilizers, thus it is easy to see that a tree of cylinder will satisfy the conclusions of the lemma.

For the other direction, note that the given conditions imply that cylinders are exactly the set of edges adjacent to a common abelian vertex. Applying the tree of cylinders construction thus yields an isomorphic tree.
\end{proof}

We gather some properties of the standard JSJ decomposition which we will use.

A vertex $v$ in an abelian splitting $L=\pi_1(\Lambda)$ is \emph{rigid} if $G_v$ is elliptic in every abelian splitting of $L$ relative to $\Gamma$ and to non-cyclic abelian groups. It is \emph{flexible} otherwise.

The vertex $v$ is \emph{of abelian type} if its vertex group $G_v$ is abelian.
It is \emph{of surface type} (see QH with trivial fiber in \cite{Sela_structure,GL_JSJ}) if $G_v$
can be written as the fundamental group of a surface $\Sigma$ so that each incident edge group
is conjugate to a subgroup of the fundamental group of a boundary component of $\Sigma$
(orbifolds are allowed in \cite{GL_JSJ}, but they don't occur here because the groups we consider are torsion-free).

\begin{lem} \label{lem_prop_std_JSJ} Let $L$ be a $\Gamma$-group which is also a $\Gamma$-limit group.
A standard JSJ decomposition $\Lambda$ for $L$ satisfies the following
\begin{enumerate}
    \item \label{it_vertex_types} non-abelian flexible vertex groups are of surface type;
    \item \label{it_UE_in_QH} if $H \subset L$ is $(\cala,\calh)$-universally elliptic 
    and is contained in a surface type vertex group of $\Lambda$, it is contained in a boundary component;
    \item \label{it_base_vertex} the vertex group containing $\Gamma$ is rigid;
    \item \label{it_bip} any edge $e$ with nontrivial $G_e$ joins a vertex with abelian group to a vertex with non abelian group;
    \item \label{it_maxab_edge} for any non-abelian vertex $v$, and incident edge $e$, if $G_e$ is nontrivial it is maximal abelian (and thus malnormal) in $G_v$;
    \item \label{it_malnormal_ab} for any non-abelian vertex $v$, and every pair of edges $e\neq e'$  with non-trivial stabilizer incident on $v$, $G_e$ is not $G_v$-conjugate to $G_{e'}$.
\end{enumerate}
\end{lem}

\begin{proof} Item \ref{it_vertex_types}
is given by Theorem 9.5 of \cite{GL_JSJ}.
Item \ref{it_UE_in_QH} appears as point (3) of Theorem 9.4 of \cite{GL_JSJ}, and implies 
item \ref{it_base_vertex} by considering the case where $H=\Gamma$ (recall that we assume that $\Gamma$ is a non-elementary hyperbolic group). 
Items \ref{it_bip}, \ref{it_maxab_edge} and \ref{it_malnormal_ab} follow from Lemma \ref{lem_cylindres}.
\end{proof}

The following result, which follows from Proposition 4.15 of \cite{GL_JSJ} characterizes (maybe not standard) JSJ trees among all $({\cal A}, \calh)$-tree.

\begin{lem} \label{lem_carac_JSJ} 
Let $\Delta$ be a splitting of $L$ corresponding to an $({\cal A}, \Gamma)$-tree.
Suppose the following hold:
\begin{itemize}
    \item edge groups of $\Delta$ are $(\cala,\calh)$-universally elliptic;
    \item a vertex group $G_v$ of $\Delta$ is either of surface type  
    or rigid relative to the incident edge groups $G_e$ (that is, does not admit any non trivial 
    $({\cal A}, \calh_{|G_v}\cup\{ G_e \}_{o(e)=v})$-splitting, 
    where $\calh_{|G_v}$ is the set of conjugates of $H$ contained in $G_v$).
\end{itemize}
Then $\Delta$ is a JSJ decomposition for $L$. 

\end{lem}

\section{Basic facts about algebraic groups} \label{sec_basic_alg_groups}

\subsection{Algebraic groups over $\Gamma$.} 
Note that if $V\subset \Gamma^d$ is a variety, then $V\times V$ is a variety in $\Gamma^{2d}$.

\begin{dfn}
An algebraic group $(V,\mu)$ over $\Gamma$ is an algebraic variety $V$
together with a group law $\mu:V\times V\ra V$ which is an algebraic map.
\end{dfn}
Note that we could have required in addition that the inverse map $V\ra V$, defined by $g\mapsto g\m$, be an algebraic map. In fact, when $\Gamma$ is torsion-free hyperbolic and $V$ is irreducible, our classification result shows that the inverse map is automatically algebraic.

We denote the group law on $V$ by $\fois$, the inverse of an element $v\in V$ by $\inv v$, and the identity element by $e$. For any $v \in V$, left and right multiplication by $v$ give algebraic automorphisms $\lambda_v:V\ra V, x\mapsto v\fois x$ and $\rho_v:V\ra V, x\mapsto x\fois v$.  

For any $v,w \in V$ we have by associativity $\lambda_v \circ \lambda_w = \lambda_{v \fois w}$ but $\rho_v \circ \rho_w= \rho_{w \fois v}$. 
We also note that 
$\rho_{\inv v}=\rho_v\m$ and
$\lambda_{\inv v}=\lambda^{-1}_{v}$.

This induces automorphisms 
$$\left\{\begin{array}{rcl}
 \rho_{v}^*:L_V&\ra& L_V \\
 f&\mapsto& f\circ \rho_v
\end{array}\right.
\text{\quad and \quad}
\left\{\begin{array}{rcl}
 \lambda_{v}^*:L_V&\ra& L_V \\
 f&\mapsto& f\circ \lambda_v
\end{array}\right. .
$$
and therefore gives two representations of the algebraic group $V$ into $\Aut_\Gamma(L_V)$ by
$$\left\{\begin{array}{rcl}V&\ra&\Aut_\Gamma(L_V)\\ v &\mapsto& \rho^*_{v}\end{array}\right.
\text{\quad and \quad}
\left\{\begin{array}{rcl}V&\ra& \Aut_\Gamma(L_V)\\ v &\mapsto& (\lambda^*_{v})^{-1}\end{array}\right.$$

\subsection{The group law formula}\label{sec_group_law}

Let $V$ be an algebraic group. Any point $v\in V$ corresponds to a $\Gamma$-homomorphism $h_v\in \HOM$.
We wish to understand the group law when points of $V$ are interpreted as elements of $\HOM$.
We denote by $e$ the identity element of $V$ and by $h_e\in \HOM$ the corresponding $\Gamma$-homomorphism.

\begin{lem}\label{lem_hlambda}
For all $v,v'\in V$, one has
$$h_{v\fois v'}=h_v\circ \rho_{v'}^*.$$

In particular, for all $v\in V$, one has $h_v=h_e\circ \rho_v^*$.

\end{lem}

\begin{proof}
For any $f\in L_V$ (viewed as a function on $V$), 
$h_{v}\circ \rho_{v'}^*(f)=h_v(f \circ \rho_{v'})=f \circ \rho_{v'} (v)=f(v\fois v')=h_{v\fois v'}(f)$.
\end{proof}

As an immediate consequence of the two equalities in Lemma \ref{lem_hlambda} 
we get
\begin{cor}\label{cor_magique} 
For all $v,v'\in V$ we have
$$h_{v\fois v'}=
h_v\circ\rho_{v'}^*
=h_e\circ\rho_{v}^*\circ\rho_{v'}^*$$
\end{cor}

The following is a straightforward but useful consequence of the transitivity of the action of $G$ on itself
by right multiplication.
\begin{lem} \label{lem_discriminating_rho}
    If $V$ is an algebraic group, then for all $h,h'\in \HOM$, there exists $v\in V$ such that
    $h'=h\circ\rho_v^*$.
    
    In particular, if $V$ is an irreducible variety which is also 
    an algebraic group, then for any $h\in \HOM$ the family of $\Gamma$-homomorphisms $\{h\circ\rho_v^* \mid v\in V\}$ is discriminating.
\end{lem}

\begin{proof} 
Using the correspondence between $V$ and $\HOM$, the transitivity of the right action of $V$ on itself
translates into the first assertion.

The second part follows from the fact that the family of all morphisms in $\HOM$ is discriminating
because $V$ is an irreducible variety 
and $L_V$ is a $\Gamma$-limit group. 
\end{proof}

\subsection{Connected components of an algebraic group.} \label{sec_irreducible}
Let $(V, \mu)$ be an algebraic group over $\Gamma$. 
The following easy remarks are completely standard.

Let $V= V_1 \cup \ldots \cup V_s$ be the unique (minimal) decomposition of $V$ as a finite union of irreducible varieties. 
Then the irreducible components are disjoint. 
Indeed, since the action of $V$ on itself is transitive, 
all the points in $V$ are contained in the same number $k$ of irreducible components $V_j$. If $k>1$ we have $V_1=(V_1 \cap V_2) \cup \ldots \cup (V_1 \cap V_s)$, which contradicts the irreducibility of $V_1$. 

It follows that the decomposition of $V$ into irreducible components is also its decomposition into connected components
for the Zariski topology. In particular, an algebraic group is connected if and only if its underlying variety is irreducible. 

The connected component containing $e$ is a finite index algebraic subgroup of $V$.

\begin{lem}\label{lem_fidx_dense}
Let $V$ be a \conn\ algebraic group, and let $V_0\subset V$ be a finite index (not necessarily algebraic) subgroup of $V$. 
Then $V_0$ is dense in $V$.
\end{lem}

\begin{proof}
If $W$ is a subvariety of $V$ containing $V_0$,
then finitely many translates of $W$ cover $V$. Since $V$ is irreducible, $W=V$.
\end{proof}

\section{Homogeneous varieties}\label{sec_gp_functions_homogeneous}
Algebraic groups are in particular homogeneous varieties. 
In this section, we collect a few consequences of homogeneity.

\begin{dfn}
A variety $V$ is \emph{homogeneous} if the group of algebraic automorphisms $V\ra V$ 
acts transitively on $V$.
\end{dfn}

\begin{qst}
What can we say about homogeneous varieties in general?
Can we classify them ?
\end{qst}

\subsection{Homogeneity}

The following is the reformulation of the definition of homogeneity using the correspondence $V \leftrightarrow \Hom_{\Gamma}(L_V, \Gamma)$.
\begin{lem}\label{lem_homogeneity_in_Hom}
$V$ is homogeneous if and only if 
for all $h,h'\in \Hom(L_V,\Gamma)$, there exists $\alpha \in \Aut_\Gamma(L_V)$ with $h\circ\alpha = h'$.\qed
\end{lem}

We saw in Lemma \ref{lem_discriminating_V0} that $V$ is irreducible as an algebraic variety if and only if 
the family $\Hom_{\Gamma}(L_V,\Gamma)$ is discriminating. Together with the previous lemma this gives
\begin{cor}\label{cor_discriminating}
    If $V$ is a homogeneous irreducible variety, for every $h\in \HOM$, the family 
    $\{h\circ \alpha \mid \alpha\in \Aut_\Gamma(L_V)\}$ is discriminating.\qed
\end{cor}

Note that Lemma \ref{lem_discriminating_rho} gives a more precise version of this result in the case where $V$ is an algebraic group.  

The following lemma says in particular that Corollary \ref{cor_discriminating} still holds if one replaces 
$\Aut_\Gamma(L_V)$ by a finite index subgroup.

\begin{lem}\label{lem_indice_fini}
    Let $h\in\HOM$ and let $M$ be a subgroup of $\Aut_\Gamma(L_V)$.
    Assume that $\{h\circ \alpha \mid \alpha\in M\}$ is a discriminating family.
    
    Then for any finite index subgroup $M_0$ of $M$, the subfamily
    $\{h\circ \alpha \mid \alpha\in M_0\}$ is still discriminating.
\end{lem}

\begin{proof}
    Write $M=M_0\phi_1 \dunion \dots\dunion  M_0\phi_n$ for some $\phi_1,\dots,\phi_n\in M$.
    Fix a finite set $\{g_1, \ldots, g_r\} \subseteq L_V$.  
    Since $\{h\circ\alpha|\alpha\in M\}$ is discriminating,
    there exists $\alpha\in M$ such that $h\circ\alpha$ does not kill any $\phi\m_i(g_j)$. 
    Let $k$ be such that $\alpha = \alpha_0 \phi_k$ for some $\alpha_0 \in M_0$. 
    Then $h \circ \alpha_0$ does not kill any $g_j$.
\end{proof}

\subsection{An example of homogeneous variety} \label{sec_ex_homogeneous_variety}

Let $\Gamma$ be a torsion-free hyperbolic group and let $c\in \Gamma\setminus\{1\}$ be such that 
$\grp{c}$ is a maximal cyclic group of $\Gamma$ and
such that any endomorphism of $\Gamma$ fixing $c$ is an automorphism (such a $c$ is known as a \emph{test element} in $\Gamma$).
Then the double $L_V=\Gamma*_c\bar\Gamma$ is a $\Gamma$-limit group, 
and it is the function group  of a homogeneous variety: for any $h \in \HOM$ we have that $h_{|\bar \Gamma}$ 
sends $\bar c$ to $c$, so it must be an isomorphism. 
Therefore $h=r \circ \sigma$ where $r$ is the canonical retraction and $\sigma \in \Aut_{\Gamma}(L_V)$.  
 This shows that $\Aut_\Gamma(L_V)$ acts transitively
on $V=\HOM$.

For instance, one can take for $\Gamma$ the free group $\Gamma=\bbF_{2g}=\grp{a_1,\dots,a_{2g}}$, and  let $c$ be the product of commutators $[a_1,a_2]\dots[a_{2g-1},a_{2g}]$. 
In this case, the double $L_V$ is the fundamental group of an orientable surface of genus $2g$, and $\Gamma\subset L_V$ is the fundamental group of a subsurface of genus $g$ with one boundary component.
The variety $V$ itself is the subset of $\Gamma^{2g}$ defined by the equation
$$V=\{(x_1,\dots,x_{2g})| [x_1,x_2]\dots[x_{2g-1},x_{2g}]=[a_1,a_2]\dots[a_{2g-1},a_{2g}]\}.$$
The automorphism group $\Aut_\Gamma(L_V)$ is isomorphic to the mapping class
group of the surface of genus $g$ with one boundary component.
 This group acts freely on $V=\HOM$ because $\Gamma$-homomorphisms $L_V\ra \Gamma$ are injective on $\bar\Gamma$
 and $\bar \Gamma$ is $\Aut_\Gamma(L_V)$-invariant.
In particular, points of $V$ are in bijection with elements of the mapping class group.

There are very many other test elements in a free group \cite{SnopceTanushevski_asymptotic}.
In particular, there exist test elements $c$ for which $\bar \bbF_n$ is 
a rigid vertex of the amalgam $\bbF_n *_c \bar\bbF_n$ (take a test element $c$ whose cyclic reduction contains all subwords of length 3 and apply \cite{CaMa_virtual}). 
 This implies that the splitting $L_V=\bbF_n *_c \bar\bbF_n$ 
 is the JSJ decomposition of $L_V$ with $\bar F_n$ a rigid vertex group.
It follows that
the group generated by the Dehn twist along 
$c$ has finite index in $\Aut_\Gamma(L_v)$.
In particular  $\Aut_\Gamma(L_V)$ is virtually cyclic and
as above, 
the action of $\Aut_\Gamma(L_V)$ on $V=\HOM$ is free and transitive.

\subsection{Grushko decomposition for a homogeneous variety}

\begin{prop}\label{prop_freeprod}
If $V$ is homogeneous and irreducible, 
the Grushko decomposition of $L_V$ relative to $\Gamma$ has the form $L_V=L^0_V*\bbF_r$ with $\Gamma\subset L^0_V$,
and $L^0_V$ is the group of functions of a homogeneous irreducible variety.
\end{prop}

\begin{proof}
Let $L_V=L^0_V*G_1*\dots *G_n*\F_r$ be the Grushko decomposition of $L_V$ relative to $\Gamma$, with $\Gamma\subset L^0_V$, and
note that $\Aut_\Gamma(L_V)$ permutes the conjugacy classes of the factors $G_1, \dots ,G_n$.
Choose any $\Gamma$-homomorphism $h_0\in\HOM[L^0_V]$, and extend it to a $\Gamma$-homomorphism  $ h\in\HOM[L_V]$ that kills $G_1,\dots,G_n$.
Since $V$ is homogeneous, the set of $\Gamma$-homomorphisms $\{h\circ\alpha \mid \alpha\in\Aut_\Gamma(L_V)\}$ is discriminating
by Corollary \ref{cor_discriminating}.  
But all of them kill $G_1$,\dots, $G_n$.
This prevents this family from being discriminating unless $n=0$.

As a subgroup of $L_V$, $L_V^0$ is fully residually $\Gamma$. Thus by Lemma \ref{lem_discriminating_V0}, the variety
associated to $L_V^0$ is irreducible.
To prove that this variety is homogeneous, 
it is enough by Lemma \ref{lem_homogeneity_in_Hom}
to check that $\Aut_\Gamma(L_V^0)$ acts transitively on $\HOM[L^0_V]$.
Given $h_0,h'_0\in \HOM[L^0_V]$, extend them to $h,h'\in\HOM$ by killing the factor $\bbF_r$.
By homogeneity of $V$, there exists $\alpha\in \Aut_\Gamma(L_V)$ such that $h'=h\circ \alpha$.
Since $\alpha$ preserves $L^0_V$, we get that $h'_0=h_0\circ \alpha_{|L_V^0}$.
\end{proof}

\subsection{JSJ decomposition for a homogeneous variety}

\begin{prop}\label{prop_GeZ} 
Let $V$ be an irreducible variety 
and let  $\Lambda$ be a standard JSJ decomposition 
 of its group of functions $L_V$.

If $V$ is homogeneous, then all the edge groups of $\Lambda$ are cyclic 
and every $\Gamma$-homomorphism $h \in \HOM$ is injective 
in restriction to all edge groups and rigid groups of $\Lambda$.
\end{prop}

\begin{proof}
By Proposition \ref{prop_freeprod}, it is enough to show the result in the case where $L_V$ is one ended. 

Let $\Aut^0_\Gamma(L_V)$ be the subgroup of finite index of $\Aut_\Gamma(L_V)$ 
consisting of automorphisms that act as the identity on the graph $\Lambda$ 
(recalling that the canonical JSJ tree is invariant under automorphisms). 

Let $h \in \Hom_{\Gamma}(L_V, \Gamma)$. By Corollary \ref{cor_discriminating} and Lemma \ref{lem_indice_fini}, 
the family $\{h\circ\alpha \mid \alpha \in \Aut^0_\Gamma(L_V)\}$ is discriminating. 
Thus, by Remark \ref{rem_discriminating}, 
there exists a sequence of automorphisms $\alpha_i\in 
\Aut^0_\Gamma(L_V)$ such that $h\circ\alpha_i$ is a discriminating sequence.

Let $R\subset L_V$ be a rigid vertex group of $\Lambda$, and $E_1,\dots,E_p\subset R$ the collection of incident edge groups.
Since each $E_k$ is abelian, $h\circ\alpha_i(E_k)$ is cyclic.
Up to removing finitely many terms in the sequence we may assume that $h\circ\alpha_i(E_k)$ is non-trivial for all $i,k$.
Since for all $i$, $\alpha_i(E_k)$ is conjugate to $E_k$ in $L_V$,
there exist maximal cyclic groups $\grp{u_k}\subset \Gamma$ such that for all $i$, 
$h\circ\alpha_i(E_k)$ is conjugate into $\grp{u_k}$.

We now view $\Gamma$ as hyperbolic relative to $\grp{u_1},\dots,\grp{u_k}$, 
and we look at 
an action of $\Gamma$ on a space with a corresponding collection of
disjoint horoballs as in \cite{Bow_relhyp}.
Assume first that the homomorphisms $h\circ \alpha_i$ are pairwise distinct modulo conjugation in $\Gamma$.
Since $R$ is non-abelian, $h\circ \alpha_i(R)$ is non-parabolic for $i$ large enough.
By  \cite[Theorem 1.2(i)]{BeSz_endomorphisms}, 
we get a non-trivial action  of $R$
on an $\bbR$-tree relative to $E_1, \dots, E_p$.
Part $(ii)$ of the same result states that if moreover every homomorphism
$h \circ \alpha_i$ is injective, then arc stabilizers in the $\bbR$-tree are abelian. 
Actually, the proof also applies here because the homomorphisms
$h \circ \alpha_i$ form a discriminating family.

Since 
abelian subgroups of $L_V$ are finitely generated, 
the hypotheses of \cite[Theorem 5.1, Corollary 5.2]{Gui_actions} are satisfied, and this action yields an abelian splitting of $R$ relative to $E_1,\dots,E_p$.
Taking a tree of cylinders for commutation as in \cite[Proposition 6.3]{GL4}, 
one gets a non-trivial abelian splitting of $R$ relative to $E_1,\dots, E_p$ and to all non-cyclic abelian subgroups of $R$,
contradicting its 
rigidity.

Thus, up to conjugating and taking a subsequence, we 
may assume that $(h\circ\alpha_i)_{|R}$ does not depend on $i$.
Since it is discriminating, this implies that $(h\circ \alpha_i)_{|R}$ is injective, hence so is $h_{|R}$.

This shows that $h$ is injective on rigid vertex groups of $\Lambda$.
It follows that every edge group $E$ of $\Lambda$ is cyclic because
every edge of $\Lambda$ is adjacent to a rigid or surface vertex.
Since $h\circ\alpha_i$ is discriminating and $\Gamma$ is torsion-free, 
$h\circ\alpha_i$ is injective on $E$ for $i$ large enough.
Since $\alpha_i$ preserves the conjugacy class of $E$, this shows that 
$h$ is injective on $E$.
\end{proof}

\begin{lem}\label{lem_homogene} Let $V$ be a homogeneous irreducible variety over $\Gamma$.
Then a standard
JSJ decomposition $\Lambda$ of $L_V$ relative to $\Gamma$ satisfies the following:
\begin{enumerate}
    \item\label{it_ab} in each abelian vertex group, the subgroup generated by 
    incident edge groups is cyclic; 
    \item\label{it_vGamma} 
     $G_{v_\Gamma}=\Gamma$;
    \item \label{it_maxZ} 
    for every non-trivial edge group $G_e$ incident on $v_\Gamma$,
    $G_e$ is maximal cyclic in $L_V$

    \item \label{it_boucle} 
    If $e, e'$  are distinct edges incident on $v_\Gamma$, then $G_e$ does not commute with any $L_V$-conjugate of $G_{e'}$.
    In particular, for every abelian vertex $v$ adjacent to $v_\Gamma$, there is at most one edge joining $v$ to $v_\Gamma$.
\end{enumerate}
\end{lem}

\begin{proof} 
Using Proposition \ref{prop_freeprod}, we may assume without loss
of generality that $L_V$ is freely indecomposable relative to $\Gamma$,
and that $L_V=\pi_1(\Lambda)$ is the canonical JSJ decomposition of $L_V$.

Let $\Aut^0_\Gamma(L_V)$ be the subgroup of $\Aut_\Gamma(L_V)$ 
consisting of automorphisms that act as the identity on the graph $\Lambda$ and whose  restriction to each edge group is a conjugation. The edge groups of $\Lambda$ are cyclic by Proposition \ref{prop_GeZ} so this is a finite index subgroup of $\Aut_\Gamma(L_V)$.

Let $h \in \Hom_{\Gamma}(L_V, \Gamma)$. By Corollary \ref{cor_discriminating} and Lemma \ref{lem_indice_fini},
there exists a sequence of automorphisms $\alpha_i\in 
\Aut^0_\Gamma(L_V)$ such that $h\circ\alpha_i$ is a discriminating sequence.

To prove (\ref{it_ab}),
let $A$ be an abelian vertex group of $\Lambda$, and $B\subset A$ be the subgroup generated by incident edge groups. 
Each $\alpha_i$ preserves the conjugacy class of $A$ and acts by conjugation on each incident edge group. Using CSA property, it easily follows that $\alpha_i$ acts by conjugation on $B$. Thus, 
$\ker( h\circ \alpha_i)_{|B}=\ker h_{|B}$ and does not depend on $i$.
Since this is a discriminating sequence, $h_{|B}$ is injective. Since $\Gamma$ is hyperbolic, it
follows that $B$ is cyclic.

For (\ref{it_vGamma}), recall that $G_{v_\Gamma}$ denotes the vertex group that contains $\Gamma$, and 
consider any $h\in \HOM$. 
The vertex group $G_{v_\Gamma}$ is rigid by Proposition \ref{lem_prop_std_JSJ} (\ref{it_base_vertex}). 
By Proposition \ref{prop_GeZ}, $h$ is injective on $G_{v_\Gamma}\supset\Gamma$. 
On the other hand, being a $\Gamma$-morphism, $h$ is the identity on $\Gamma$, so $G_{v_\Gamma}=\Gamma$.

To prove (\ref{it_maxZ}), consider an edge $e$ incident on $v_\Gamma$. 
By Proposition \ref{prop_GeZ}, $G_e$ is infinite cyclic. 
By Lemma \ref{lem_prop_std_JSJ}(\ref{it_maxab_edge}),
$G_e$ is maximal abelian hence maximal cyclic in $G_{v_\Gamma}=\Gamma$. 
Assuming that $G_e$ is not maximal cyclic in $L_V$, 
consider $c\in L_V$ and $k>1$ such that $G_e=\grp{c^k}$.
Choose some $h\in \HOM$ (there exists one because $L_V$ is a $\Gamma$-limit group),
and view it as a retraction $L_V\ra \Gamma$.
Since $c^k\in \Gamma$, we have $c^k=h(c^k)=h(c)^k$ 
contradicting that that $G_e$ is maximal cyclic in $\Gamma$.

To prove (\ref{it_boucle}), consider two distinct edges $e,e'$ 
incident on $v_\Gamma$ whose edge groups have commuting conjugates in $L_V$. 
View $G_e$ and $G_{e'}$ as subgroups of $\Gamma$.
Using a retraction $L_V\ra \Gamma$ as above, we see that 
$G_e$ and $G_{e'}$ have commuting conjugates in $\Gamma$.
By Lemma \ref{lem_prop_std_JSJ}(\ref{it_malnormal_ab}),
this implies $e=e'$.
\end{proof}

\section{Extension of centralizers and algebraic subgroups} \label{sec_trimming}

The goal of this section is Theorem \ref{thm_norigid} showing that the JSJ decomposition associated to the group of functions $L_V$ of a \conn\ algebraic group $V$
has no rigid vertex except for $v_\Gamma$.
We will first associate to $V$ an algebraic subgroup $W_0\subset V$
whose associated limit group is obtained by replacing some abelian vertex groups by the cyclic subgroup generated
by incident edge groups. This will be used to slightly simplify some arguments including the description of $L_{V\times V}$.

\subsection{An algebraic subgroup associated to an extension of centralizers} \label{sec_extension_centralizers}

Let $V$ be a \conn\ algebraic group. Let $\Lambda$ be a standard JSJ decomposition of $L_V$ (we don't assume that $L_V$ is one-ended).

Denote by $A_1,\dots,A_q$ the collection of abelian vertex groups in $\Lambda$.
Using Lemma \ref{lem_homogene}(\ref{it_ab}), for each $i\leq q$, 
denote by $\grp{c_i}$  the smallest direct factor of $A_i$ containing all incident edge groups. 
    Let $L_W\subset L_V$ be the subgroup obtained by replacing each abelian vertex group $A_i$ by the cyclic group $\grp{c_i}$.

We will apply the following lemma to the $\Gamma$-homomorphism $h_e\in \HOM$ corresponding to
the neutral element $e\in V$.

\begin{lem}\label{lem_direct_sum}
    Let $V$ be a \conn\  algebraic group as above. 

    Then for any $h\in \HOM$, 
    each abelian vertex group 
    $A_i$ splits as $A_i=\grp{c_i} \oplus \ker h_{|A_i}$.
\end{lem}
The lemma actually holds in fact for any homogeneous and irreducible variety.

\begin{proof}
Let $h\in\HOM$ be any $\Gamma$-homomorphism.
We know by Proposition \ref{prop_GeZ} that for all $i\leq q$, 
$h$ is injective on $\grp{c_i}$
and that $h$ has cyclic image. This does not imply the result because it might happen that 
$ h(\grp{c_i})\subsetneqq h(A_i)$.

    For each $i\leq q$, 
    choose an arbitrary decomposition $A_i=\grp{c_i}\oplus A'_i$, 
    and consider the retraction $r:L_V\ra L_W$ killing $A'_i$ for all $i$.
 Now consider the $\Gamma$-homomorphism $h'=h\circ r$. 
    Note that for each $i\leq q$, $\ker h'_{|A_i}=A'_i$ because
    $h'$ is injective on $\grp{c_i}$ by Proposition \ref{prop_GeZ}.
    By homogeneity, there exists $\alpha\in\Aut_\Gamma(L_V)$
    such that $h=h'\circ\alpha$.

    Fix an abelian vertex group $A_i$. 
    Up to precomposing $h$ by an inner automorphism (which does not change its kernel), we may assume that $\alpha(A_i)=A_{j}$ for some $j\leq q$.
    Then $\ker h_{|A_i}=\alpha\m(\ker h'_{|A_{j}})=\alpha\m(A'_{j})$. Since $\alpha\m(\grp{c_{j}})=\grp{c_i}$,
    the decomposition $A_{j}=\grp{c_j}\oplus A'_{j}$ gives a decomposition 
    $A_i=\grp{c_i}\oplus (\ker h_{|A_i})$.
\end{proof}
    
For each $i\leq q$, denote by $N_i=\ker h_{e|A_i}$ where
$h_e\in \HOM$ is the $\Gamma$-homomorphism corresponding to the neutral element $e\in V$.
By Lemma \ref{lem_direct_sum},  $A_i=\grp{c_i}\oplus N_i$.

We now define the variety and limit group obtained by trimming a collection of abelian groups in $L_V$ (see Figure \ref{fig_trim}).
Given a subset $\cala\subset \{1,\dots, q\}$, 
denote by $L_{W_\cala}\subset L_V$
the subgroup obtained by replacing, for each $i\in \cala$,  the abelian vertex group $A_i$ in $\cala$ by the cyclic group $\grp{c_i}$.
Thus $L_V$ can be viewed as a multiple extension of centralizers $L_V=L_{W_\cala}*_{\grp{c_i}} (\grp{c_i}\oplus N_i)$
(indexed by $i\in \cala$).
Consider the retraction $r:L_V\ra L_{W_\cala}$
killing $N_i$ for all $i\in \cala$, and denote by $j:L_{W_\cala}\ra L_V$ the inclusion.
Viewing the set $\HOM[L_{W_\cala}]$ as an algebraic variety $W_\cala$,
the precomposition by 
$r$ and $j$ respectively yield an embedding $r^*:W_\cala=\Hom_{\Gamma}(L_{W_\cala}, \Gamma)\ra V=\HOM$
and a retraction $j^*:V\ra W_\cala$ such that 
$j^*\circ r^*=(r\circ j)^*=\id_{W_\cala}$.

\begin{figure}[ht]
    \centering
    \includegraphics[width=\linewidth]{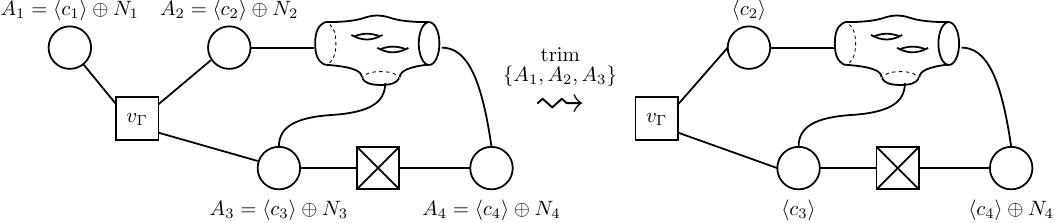}
    \caption{Trimming $\cala=\{A_1,A_2,A_3\}$. 
    Round vertices carry abelian groups, square vertices are rigid.
    The abelian vertex $\grp{c_1}$ disappears when we pass to the minimal splitting.}
    \label{fig_trim}
\end{figure}

\begin{dfn}[Trimming a collection of abelian groups]
We say that $W_\cala$ or $L_{W_\cala}$ are obtained from $V$ or $L_V$ by trimming the collection of abelian vertex groups $\cala$.
\end{dfn}

The following lemma says that the subset $r^*(W)$ is an algebraic subgroup of $V$
if the collection  of abelian vertices of $\Lambda$ corresponding to $\cala$ is invariant under $\Aut_\Gamma(L_V)$.

\begin{lem}\label{lem_algebraic_subgroup}
  Consider a \conn\ algebraic group $V$,
  and $\Lambda$ a standard JSJ decomposition of $L_V$.
 Consider a collection of abelian vertex groups $\cala=\{A_i\}$ of $\Lambda$, such that the corresponding set of vertices of $\Lambda$ is invariant under $\Aut_\Gamma(L_V)$.

Then the image of the embedding $r^*:W_\cala\ra V$ is an algebraic subgroup of $V$.
\end{lem}

\begin{proof} 
    Note that $r^*(W_\cala)$ is the set of $\Gamma$-homomorphisms $h\in \HOM$ that kill $N_i$ for all $i\in \cala$.
   In particular, it is a subvariety of $V$.
    Let $w_1, w_2 \in r^*(W_\cala)$ and $h_{w_1}, h_{w_2} \in \HOM[L_V]$ be the corresponding $\Gamma$-homomorphisms.

Let us check that $w_1\fois w_2$  is in $r^*(W_\cala)$, i.e. that $h_{w_1\fois w_2}$ kills every $N_i$.
        We have $h_{w_i}= h_e \circ \rho^*_{w_i}$ (see Lemma \ref{lem_hlambda}). 

    By Corollary \ref{cor_magique}, we get
    $$h_{w_1\fois w_2}=h_e\circ \rho_{w_1}^*\circ\rho_{w_2}^*.$$
    We know that $h_{w_1} = h_e \circ \rho^*_{w_1}$ kills $N_i$ for all $i\in \cala$, 
    it thus suffices to prove that $\rho^*_{w_2}$ permutes the conjugacy classes of the groups $N_i$. 
    So fix $i\in \cala$.
    Since $\rho^*_{w_2}$ permutes the vertices associated to $\cala$, 
    consider $A_{i'}\in \cala$ 
    such that $\rho^*_{w_2}(A_i)$ is conjugate to $A_{i'}$ and in particular, $\rho(N_i)$ is conjugate into $A_{i'}$.
    Now $h_{w_2} = h_e \circ \rho^*_{w_2}$ kills $N_i$, hence $\rho^*_{w_2}$ sends $N_i$ into a conjugate of $\ker(h_{e|A_{i'}})=N_{i'}$. 
    This shows that $\rho_{w_2}^*$ permutes the groups $\{N_i|i\in \cala\}$ up to conjugation,
    and that $h_{w_1\fois w_2}$ kills every $N_i$ for $i\in \cala$. 
    Thus $h_{w_1\fois w_2}\in r^*(W_\cala)$ which concludes the proof that $r^*(W_\cala)$ is stable under multiplication.
    
    Since $\rho^*_{\inv w}=(\rho^*_w)\m$ permutes the $N_i$'s up to conjugation, 
    $h_{\inv w}=h_e\circ(\rho^*_w)\m$ kills every $N_i$ which shows that
    that $r^*(W_\cala)$ is stable under taking inverse.
\end{proof}

\begin{lem} \label{lem_trimA0}
Let $V$ be a \conn\ algebraic group $V$,
  and let $\Lambda$ be a standard JSJ decomposition of $L_V$.  
  Let $\cala_0$ 
  be the collection of 
  non-cyclic abelian vertex groups of $\Lambda$ adjacent to $v_{\Gamma}$.

Denote by $L_{W_0}$ the group obtained from $L_V$ by trimming $\cala_0$.
Let $\Lambda_{W_0}$ 
be the decomposition
of $L_{W_0}$ obtained from $\Lambda$ by replacing, 
for each $i \in \cala_0$,
the group $A_i$ by its edge summand $\grp{c_i}$ and passing to a minimal subgraph of groups.

Then $\Lambda_{W_0}$ is a standard JSJ decomposition of $L_{W_0}$.
\end{lem}

\begin{rem}\label{rem_cyclic}
For each edge group $G_e$ incident on $v_\Gamma$, the maximal
abelian subgroup of $L_{W_0}$ containing $G_e$ is the cyclic group $G_e$ itself (see Lemma \ref{lem_homogene}(\ref{it_maxZ})).
\end{rem}

\begin{proof}
We first claim that every abelian splitting $L_{W_0}=\pi_1(\Delta)$ relative to $\Gamma$ and to non-cyclic abelian groups 
extends to a splitting $\hat{\Delta}$ of $L_V$ relative to $\Gamma$ and to its non-cyclic abelian groups.
Indeed, write $L_V$ as an extension
of centralizers $L_V=L_{W_0}*_{\grp{c_i}} A_i$.
To refine this splitting using $\Delta$, we need to check
that $c_i$ is elliptic in $\Delta$.
This is the case: $\grp{c_i}$ is conjugate into $\Gamma$ since $v_i$ is adjacent to $v_\Gamma$.
This proves the claim.

Since edge groups of $\Lambda$ are elliptic in $\hat \Delta$,
this implies that edge groups of $\Lambda_{W_0}$ are elliptic in $\Delta$.
This shows that $\Lambda_{W_0}$ is universally elliptic.

To deduce that $\Lambda_{W_0}$ is an abelian JSJ decomposition of $L_{W_0}$ relative to $\Gamma$ and to its non-cyclic abelian subgroups, it suffices to check that for each vertex $v$ of $\Lambda_{W_0}$, $G_v$ has a trivial JSJ decomposition relative to $\Gamma$ and to its incident edge groups
by Lemma \ref{lem_carac_JSJ}.
This is clear for abelian vertices, and also for non-abelian vertices having the same
incident edges in $\Lambda_{W_0}$ and $\Lambda$. 

It may happen that some edges of $\Lambda$ are not in $\Lambda_{W_0}$ because of the need to pass to the minimal splitting. To make the graph of groups minimal, one removes edges $e$ with a valence $1$ endpoint $v$ such that $G_e=G_v$, and repeat this operation. In our situation, 
these edges are exactly the edges of $\Lambda$
joining $v_\Gamma$ to a terminal abelian vertex. Once these edges have been deleted, then
the graph of groups is minimal because the vertex group at $v_\Gamma$ is non-abelian.
We conclude that the only non-abelian vertex of $\Lambda_{W_0}$ whose peripheral structure
may have changed is $v_\Gamma$.
Since we are looking at splittings relative to $\Gamma$,
this concludes the proof that $\Lambda_{W_0}$ 
is a JSJ decomposition.

One easily checks that $\Lambda_{W_0}$ is actually a standard JSJ decomposition of $L_{W_0}$
by applying Lemma \ref{lem_cylindres} to the complement of the set of edges with trivial stabilizer.
\end{proof}

\subsection{Coordinate group of the double}

The goal of this subsection is to show that 
$L_{W_0\times W_0}$ isomorphic to $L_{W_0}*_\Gamma L_{W_0}$
(Corollary \ref{cor_double_trim}).
Note that this would not be true if we had not trimmed abelian groups neighbouring $v_\Gamma$.
In view of Lemma \ref{lem_direct_product},
we rely on the following lemma.

\begin{lem}\label{lem_double}
   Let $L$ be the function group of an irreducible variety $W$. 
   Assume that $L$ has a cyclic splitting $\Theta$ such that
   \begin{enumerate}
       \item there is a vertex $v_\Gamma$ with vertex group $\Gamma$
    \item there is no edge with non-trivial stabilizer joining $v_\Gamma$ to itself 
    \item $\Theta$ is $1$-acylindrical: every non-trivial element fixes at most one edge.
   \end{enumerate}
Then the coordinate group of $W\times W$ is the double $L*_{\Gamma} \bar L$
where $\bar L$ is an isomorphic copy of $L$.
\end{lem}

\begin{proof} 
Denote by $D$ the double $L*_\Gamma \bar L$.
By Lemma \ref{lem_direct_product},
the coordinate group of $W\times W$ is the 
maximal residually $\Gamma$ quotient of $D$.
Thus, it suffices to prove that $D$ is a $\Gamma$-limit group.

For any $g\in L$, denote by $\bar g$ its copy in $\bar L\subset D$.
Let $\phi:L*_\Gamma \bar L\ra L$ be the morphism that is the identity on the left factor $L$,
and maps $\bar g\in \bar L$ to $g\in L$.

One can write $D$ as a the fundamental group of the graph of groups $\Theta^2$
obtained by gluing $\Theta$ with a copy $\bar \Theta$ of itself along the vertex $v_\Gamma$.
One easily checks that $\Theta^2$ is $2$-acylindrical:
any segment of length 3 in the Bass-Serre tree of $\Theta^2$
has a subsegment of length $2$ in a Bass-Serre tree of $\Theta$ or $\bar \Theta$ by Assumption 2, so one concludes using $1$-acylindricity of $\Theta$ and $\bar \Theta$.

Non-trivial edge groups of $\Theta$ are maximal abelian in $L$ by $1$-acylindricity of $\Theta$.
By \cite[Corollary A.8]{CG_compactifying}, $D$ is CSA.

The morphism $\phi$ is 
injective on all 
vertex groups of $\Theta^2$.
Since $\Theta^2$ is a graph of groups whose non-trivial edge groups are maximal abelian in neighbouring vertex groups, and since $L$ is a $\Gamma$-limit group, 
it follows that $D$ is a $\Gamma$-limit group by \cite[Theorem 1.31]{Sela_diophantine7}, or \cite[Prop 4.11]{CG_compactifying} which adapts immediately to the context of $\Gamma$-limit groups.
\end{proof}

\begin{cor}\label{cor_double_trim}
Let $V$ be a \conn\ algebraic group, and $W$ its $\cala_0$-trimming as in Lemma \ref{lem_trimA0}.
    
    Then $L_{{W_0}\times {W_0}}=L_{W_0}*_\Gamma L_{W_0}$.
\end{cor}

\begin{proof}
Let $\Lambda_{W_0}$ be the decomposition of $L_{W_0}$ as in Lemma \ref{lem_trimA0}.
    Let $\Theta$ be the graph of groups obtained from $\Lambda_{W_0}$
    by collapsing all edges that are not adjacent to $v_\Gamma$.

We apply Lemma \ref{lem_double}.
By Lemma \ref{lem_homogene} (\ref{it_vGamma})
and recalling that $\Lambda_{W_0}$ is bipartite,
we get that the first two assumptions of the lemma hold. 
Thus it suffices to check that $\Theta$ is $1$-acylindrical.

To prove this
consider
 $e\neq e'$ two edges with common vertex $w$ in the Bass-Serre tree of $\Theta$ and with $G_e\cap G_{e'}\neq \{1\}$. 
 Then $G_e=G_{e'}$ by Remark \ref{rem_cyclic}. 
The two edges $e$ and $e'$ are not in the same $G_w$-orbit because
being maximal abelian in $L_{W_0}$, $G_e$ is malnormal in $G_w$. 
Downstairs, the image of $w$ need not be equal to $v_\Gamma$
but the images of $e$ and $e'$ in $\Theta$
are two distinct edges incident on $v_\Gamma$ 
whose edge groups are conjugate into $L_{W_0}$, 
contradicting Lemma \ref{lem_homogene} (\ref{it_boucle}).
This shows that $\Theta$ is $1$-acylindrical and concludes the proof.
\end{proof}

\subsection{Rigid vertices in the JSJ decomposition}

\begin{thm}\label{thm_norigid} Let $V$ be a \conn\ algebraic group, 
and $\Lambda_V$ be a standard JSJ decomposition of $L_V$.

Then $\Lambda_V$ has no non-abelian rigid vertex other than the vertex $v_\Gamma$ with vertex group $\Gamma$, and no vertex of surface type whose corresponding surface has finite mapping class group.
\end{thm}

\begin{rem}
The only hyperbolic surfaces with finite mapping class group
are the pair of pants and the twice punctured projective plane (see e.g.~\cite[\S 5.1.4]{GL_JSJ}). 
A vertex of surface type whose underlying surface is a pair of pants
is actually rigid; but this is not not the case for the
 twice punctured projective plane
 which has splittings dual to curves bounding a Möbius band.
\end{rem}

\begin{proof}[Proof of Theorem \ref{thm_norigid}]
    Let $\Lambda$ be a JSJ decomposition of $L_V$. 
    
    Let $L_{W_0}$ and $\Lambda_{W_0}$ be obtained by trimming $\cala_0$ as in Lemma \ref{lem_trimA0}. 
    It suffices to show that $\Lambda_{W_0}$
    has no non-abelian rigid vertex group distinct from $v_\Gamma$.

    By Corollary \ref{cor_double_trim}, $L_{{W_0}\times {W_0}}$ is the double $D=L_{W_0}*_\Gamma \bar L_{W_0}$. 

    The multiplication law $\mu:W_0\times W_0\ra W_0$ induces a dual morphism $\mu^*:L_{W_0}\ra D$.
    Denote by $h_e:L_{W_0}\ra \Gamma$ the $\Gamma$-homomorphism corresponding to the identity element $e\in {W_0}$.
    The map $\Psi:w\in {W_0}\mapsto (e,w)\in {W_0}\times {W_0}$ induces a dual morphism $\Psi^*: D \ra L_{W_0}$.
    It coincides with $h_e$ on $L_{W_0}$, and maps $\bar g\in \bar L_{W_0}$ to $g$.
    We will also use the symmetric map $\Psi':w\mapsto (w,e)$.
    
    Since $\mu\circ\Psi=\id_{W_0}$, we get $\Psi^*\circ \mu^*=\id_{L_{W_0}}$ so $\mu^*$ is injective.
    Let $\Lambda_D$ be the splitting of $D$ obtained by gluing two copies $\Lambda_{W_0},\bar\Lambda_{W_0}$ of $\Lambda_{{W_0}}$ along $v_\Gamma$.       
    
    Fix a non-abelian rigid vertex group $R$ of $\Lambda_{W_0}$. The monomorphism $\mu^*$ gives an action of $L_{W_0}$ on the Bass-Serre tree of $\Lambda_D$. In this action, $R$ is elliptic thus
    $\mu^*(R)$ is conjugate into a vertex group $G_u$ of $\Lambda_D$.
    If $u$ corresponds to a vertex coming from $\Lambda_{W_0}$ (allowing $u=v_\Gamma$),
    then $\Psi^*(\mu^*(R))$ is conjugate into $h_e(G_u)\subset \Gamma$.
    Since $\Psi^*\circ\mu^*=\id_{L_{W_0}}$, we get that $R$ is conjugate into $\Gamma$.
    If $u$ comes from $\bar\Lambda_{W_0}$ we proceed similarly using $\Psi'^*$ instead of $\Psi^*$.

This rules out all non-abelian rigid vertex groups except $v_\Gamma$, 
including surface type vertices whose underlying surface is a pair of pants.

There remains to prove that there is no surface-type vertex whose underlying
surface is a twice punctured projective plane. If $R$ is such a vertex group, the argument above apply if $\mu^*(R)$ is elliptic in $\Lambda_D$. 
If not, the splitting induced on $R$ 
is dual to a curve bounding a Möbius band, hence corresponds to
an amalgamated product $R' *_{\grp{z^2}} \grp{z}$.
The argument above shows that $R'$ lies in $\Gamma$ up to conjugacy. 
So we may assume that $R' \subset \Gamma$ and in particular 
$z^2\in \Gamma$. By Lemma \ref{lem_homogene}(3),
we get that $z\in \Gamma$, so $R\subset \Gamma$.
This concludes the proof of the theorem.
\end{proof}

\section{Bounded stretch and application to vertices of surface type} \label{sec_bounded_stretch}

\subsection{Length and stretch}

Let $L$ be a $\Gamma$-group with a finite $\Gamma$-generating set $X_1, \ldots, X_n$, \ie such that $L=\grp{\Gamma,X_1,\dots,X_n}$. 
\begin{dfn}
Given $f \in L$, we define the length $\ell (f)$ of $f$ to be the minimal length of a word on the alphabet $\{X^{\pm 1}_1, \ldots, X^{\pm 1}_n\}\cup (\Gamma\setminus\{1\})$ representing $f$.
\end{dfn}

\begin{rem}
The length of $f$ should be thought of as an analogue of the degree of a polynomial
(although the length of $aX_1X_2b X_2^{-2}$ is 6 while its degree should rather be 4).
\end{rem}

\begin{dfn} Let $\alpha \in \Aut_{\Gamma}(L)$. We define the \emph{stretch} of $\alpha$ 
(with respect to the $\Gamma$-generating set $X_1,\dots,X_n$ of $L$)
as
$$\Lip (\alpha) = \sup_{f \in L\setminus\{1\}} \frac{\ell(\alpha(f))}{\ell(f)}$$
$$=\max \{1, \ell(\alpha(X_1)),\dots,\ell(\alpha(X_n))\}$$
\end{dfn}

\begin{rem}
The equivalence between the two definitions follows from the fact
that $\Gamma\cup \{X_1,\dots,X_n\}$ is a generating set of $L$,
so the stretch of $\alpha$ is bounded by the length of 
the images of these elements.
\end{rem}

We say that an algebraic group $G$ acts algebraically on a variety $V$ if the
map $G\times V\ra V$ is an algebraic map.
Given $g\in G$, we denote by $\tau_g:v\in V\mapsto g.v\in V$,
and by $\tau_g^*\in\Aut_\Gamma(L_V)$ the automorphism
sending $f\in L_V$ to $f \circ \tau_g: v\mapsto f(g.v)$.
For example, if $V$ is an algebraic group, we will consider the action of $G=V$ on itself by right multiplication, and of $G=V \times V$ on $V$ by two-sided multiplication.

The following lemma is a key observation. 
It is analogous to the classical fact of algebraic geometry saying that
if $\bbG\actson M$ is an algebraic action of an algebraic group on an algebraic variety $M$,
for every function $f$ on $M$, the $\bbG$-orbit of $f$ spans a finite dimensional vector space
(\cite{Brion_luminy}).
If $M$ is the affine space, this means that the $\bbG$-orbit of a polynomial $f$ has bounded degree.

\begin{lem}[Bounded stretch]\label{lem_alg_action}
    Let $G$ be an algebraic group acting algebraically on an algebraic variety $V$. 
    
    Then
    there exists a constant $C>0$ such that for all $g\in G$, $\Lip(\tau_g^*)\leq C$.
    
    Equivalently, for every function $f\in L_V$, and every $g\in G$, $\ell(f\circ \tau_g)\leq C\ell(f)$.
\end{lem}

\begin{proof}
View $G$ (resp $V$) as a subset of $\Gamma^n$ (resp $\Gamma^p$).
Given $g\in G$ (resp.~$v\in V$), we will write $g=(g_1,\dots,g_n)$ (resp.\
$v=(v_1,\dots,v_p)$ 
its coordinates in $\Gamma^n$ (resp. $\Gamma^p$).
Denote by $X_1,\dots,X_p:V\ra \Gamma$ the 
coordinate functions -- sending $(v_1,\dots,v_p)\in V$ to $v_i$.

The action map $\alpha:G\times V\ra V$ can be written as $(\alpha_1,\dots,\alpha_p)$
where each $\alpha_i$ is a word map on $n+p$ variables
associated to some $\Gamma$-word $\tilde \alpha_i\in \Gamma*\grp{z_1,\dots,z_n,x_1,\dots, x_p}$. In other words,
$$\alpha_i:(g_1,\dots,g_n,v_1,\dots,v_p)\mapsto \tilde \alpha_i(g_1,\dots,g_n,v_1,\dots,v_p).$$

Given $g\in G$, the map $\tau_g:V\ra V$ is defined by $\tau_g(.)=\alpha(g,.)$.
Thus, 
$\tau_g^*(X_i)=X_i\circ\tau_g=X_i(\alpha(g,.))=\alpha_i(g,.)=
\tilde \alpha_i(g_1,\dots,g_n,X_1,\dots,X_p)\in L_V$.
Since by definition the elements of $g_i\in\Gamma$ have length at most 1 in $L_V$,
$\ell(\tau_g^*(X_i))$ is bounded by the length  of $\tilde \alpha_i$ as a $\Gamma$-word in $n+p$ variables.
This bound does not depend on $g$, which proves the lemma.
\end{proof}

Given an element $g$ acting isometrically on a metric space $Y$, 
we denote by 
$||g||_Y=\inf_{y\in Y} d(y,gy)$ its translation length.
By the triangle inequality, we obtain:
\begin{lem}\label{lem_bound_algebraic}
     Let $G$ be an algebraic group acting algebraically on an algebraic variety $V$. Consider an action of $L_V$ by isometries on a metric space $Y$, such that $\Gamma$ has bounded orbits.

Then there exists a constant $K$ such that
for all $f\in L_V$ and 
 $g\in G$, $||\tau_g^*(f)||_Y \leq K \ell(f)$.
\end{lem}

\begin{proof}
Endow $L_V=\grp{\Gamma,X_1,\dots,X_p}$ with the word metric $\ell$ on the infinite generating set $\{X_1,\dots,X_p\}\cup\Gamma\setminus\{1\}$. 
Fix a basepoint $* \in Y$. The orbit map $L_V \to Y$ given by $f \mapsto f.*$ is $M$-Lipschitz
where $M$ is an upper bound on the diameter of $\Gamma.*$ and on the distances $d_Y(*,X_i.*)$. 
Thus, using Lemma \ref{lem_alg_action}, for all $f\in L_V$ and $g\in G$,
$$d_Y(*,f\circ \tau_g.*)\leq M \ell(f\circ \tau_g)\leq MC\ell(f).$$
\end{proof}

If $V$ is an algebraic group, we can apply the above result with $G=V$ acting on itself
by right multiplication,  
i.e. $G \times V \to V$ is given by $(g,v)\mapsto v \fois g$. 
In this case $\tau_g = \rho_{g}$ so we get
\begin{cor}\label{cor_bound} Let $V$ be an algebraic group.
Consider an action of $L_V$ by isometries on a metric space $Y$, such that $\Gamma$ has bounded orbits.

Then there exists a constant $K$ such that 
for all $f\in L_V$, and all $v\in V$, $||\rho_v^*(f)||_Y \leq K\ell(f)$.
\end{cor}

We will also use Lemma \ref{lem_bound_algebraic} for the action 
of the algebraic group $G=V \times V$ on the variety $V$ by left and right multiplication
given by $(a,b) \cdot x = a \fois x \fois b$
(formally, this is rather a left action of $V\times \check V$ on $V$
where $\check V$ is the algebraic group $V$ with the multiplication law $\check\mu$
defined by $\check\mu(\ul v_1,\ul v_2)=\mu(\ul v_2,\ul v_1)$).

For $a,b \in V$ we denote by $\Delta_{a,b}: V \to V; v \mapsto a \fois v \fois b$. In this setting Lemma \ref{lem_bound_algebraic} 
gives:

\begin{cor}\label{cor_bound2} Let $V$ be an algebraic group.
Consider an action of $L_V$ by isometries on a metric space $Y$, such that $\Gamma$ has bounded orbits. 

Then for all $f\in L_V$, there exists a constant $C_f$ such that 
for all $a,b\in V$, $||\Delta_{a,b}^*(f)||_Y \leq C_f$. \qed
\end{cor}

\subsection{Surfaces in the JSJ decomposition} 

The goal of this section is the following result.
\begin{thm} \label{thm_noQH}
Let $V$ be a \conn\ algebraic group with coordinate group $L_V$,
and $\Lambda$ be a standard JSJ decomposition of $L_V$.

Then $\Lambda$ has no vertex of surface type whose corresponding surface has infinite mapping class group.
\end{thm}

We will use the following lemma which is standard in the
orientable case. See below for a proof in general.

\begin{lem}\label{lem_MCG}
Let $\sigma$ be an element of infinite order in
the mapping class group of a hyperbolic surface $\Sigma$ (maybe with boundary).

Then there exists an essential simple closed curve 
(maybe 1-sided if $\Sigma$ is non-orientable) $\beta$
such that the intersection numbers $i(\beta,\sigma_v^n(\beta))$ are unbounded.
\end{lem}

\begin{proof}[Proof of Theorem \ref{thm_noQH}]
Let $Q$ be a vertex group of surface type of $\Lambda$.
Let $V_0\subset V$ be the finite index subgroup consisting of elements
$v$ such that $\rho_v^*$ preserves the conjugacy class of $Q$ and of incident edge groups
(we don't claim that $V_0$ is an algebraic subgroup of $V$).

Given $v\in V_0$, 
consider $g_v\in L_V$ such that $\ad_{g_v}\circ \rho_v^*(Q)\subset Q$, 
and consider the restriction $\sigma_v=(\ad_{g_v}\circ \rho_v^*)_{|Q}$.
Noting that the class of $\sigma_v$ in $\Out(Q)$ does not depend on the choice of $g_v$,
the map $v\mapsto \sigma_v$ is a well defined morphism $\sigma:V_0\ra \Out(Q)$.
We view $\sigma_v$ as an element of the mapping class group of the underlying surface $\Sigma$.

We claim that this morphism has finite image.
Otherwise, since $\Out(Q)$ is virtually torsion-free, there exists $v\in V$ such that $\sigma_v$ has infinite order.

Consider $\beta$ a simple closed curve given by Lemma \ref{lem_MCG}.

Let $T_\beta$ be the refinement of the JSJ decomposition of $L_V$ given by the curve $\beta$ (if $\beta$ is one-sided, this is the splitting dual to the boundary of a Möbius band with core curve $\beta$).
Let $g\in Q$ be an element representing $\beta$. 
Then $||\sigma_v^n(g)||_{T_\beta}= i(\beta,\sigma_v^n(\beta))$ is unbounded.
Since $\sigma_v^n(g)$ equals $(\rho_v^*)^n(g)$ up to conjugacy, and $(\rho_v^*)^n=\rho^*_{v^{\fois n}}$, this contradicts Corollary \ref{cor_bound} for the space $Y=T_\beta$.
This concludes the proof that $\sigma$ has finite image.

Let $h\in \HOM$ be any $\Gamma$-homomorphism. We claim that $h_{|Q}$ is injective.
Indeed, let $V_1\subset V_0$ be the finite index subgroup
of $V$ consisting of elements $v\in V_0$ such that $\sigma_v=\id_Q$ in $\Out(Q)$.
By Lemma \ref{lem_indice_fini}, the family $\{h\circ\rho_v^*| v\in V_1\}$
is discriminating. Since all these morphisms differ by a conjugation on $Q$,
$h_{|Q}$ is injective.

Now assume by contradiction that $Q$ has infinite outer automorphism group.
This implies that it contains a Dehn twist of infinite order.
Extend this Dehn twist to an automorphism
$\alpha \in Aut_{\Gamma}(L_V)$ preserving $Q$,
and acting as the identity on a non-abelian subgroup $Q_1\subset Q$ 
(corresponding to the fundamental group of a connected component of the complement of the twisting curve).
By homogeneity, and since $V_1$ has finite index in $V$, the $h \circ \alpha^k$ fall into finitely many $V_1$-orbits. Thus we can find 
distinct indices $k,l \in \N$  
such that $h\circ \alpha^k$ and $h\circ\alpha^l$
are in the same $V_1$-orbit, \ie there exists $v \in V_1$ such that $h \circ \alpha^k = h \circ \alpha^l \circ \rho^*_{v}$. 

Since $\rho^*_{v}$ restricts to a conjugation on $Q$, 
there exists $c\in \Gamma$ such that
$$h_{|Q} \circ (\alpha^k)_{|Q} = \ad_c \circ h_{|Q} \circ  (\alpha^l)_{|Q}$$
Since $\alpha_{|Q_1}=\id$, $c$ centralizes $h(Q_1)$, which is non abelian by injectivity of $h$ on $Q$. This implies $c=1$ since $\Gamma$ is torsion-free hyperbolic.
Since $h$ is injective on $Q$, $\alpha^k_{|Q}=\alpha^l_{|Q}$, a contradiction.
\end{proof}

\begin{proof}[Proof of Lemma \ref{lem_MCG}.] 
In the orientable case, this classical fact follows from Thurston's classification and \cite[Prop.~3.2, Cor~.13.3, Th.~14.24]{FaMa_primer}.

In the non-orientable case, Thurston's classification still holds
\cite{Paris_beginners,Wu}. If
the decomposition associated to $\sigma$ 
contains a pseudo-Anosov piece, 
one can use that the lift of a pseudo-Anosov
to the orientation cover is still pseudo-Anosov \cite{Wu}
so one can take for $\beta$ any simple curve contained in the pseudo-Anosov piece.

Otherwise, some power of $\sigma$ is a product of Dehn twists along
disjoint two-sided simple closed geodesics $\alpha_1, \cdots, \alpha_p$. 
where no $\alpha_i$ bounds a Möbius band.
Let $\Sigma'$ be the connected component of $\Sigma\setminus (\alpha_2\cup\dots\cup \alpha_p)$ containing $\alpha_1$.

We claim that
there exists $\beta \subset \Sigma'$ such that one of the following holds 
\begin{itemize}
    \item $\beta$ is two sided and intersects $\alpha_1$ non-trivially 
\item or $\beta$ intersects $\alpha_1$ exactly once and the regular neighbourhood of $\alpha_1\cup \beta$ is a once-punctured Klein bottle.
\end{itemize}

If $\Sigma'\setminus \alpha_1$ has two connected components $\Sigma'_1,\Sigma'_2$, 
then one constructs $\beta$ satisfying (1)
by concatenating two essential arcs $\gamma_i \subset \Sigma'_i$ for which $\alpha_1\cup \gamma_i$ has
an orientable neighbourhood. The existence of $\gamma_i$ is clear 
if $\Sigma'_i$ is orientable. If $\Sigma'_i$ is non-orientable, then
it can be written as a $n$-punctured sphere (having $\alpha_1$ as one of its 
boundary components), to which one glues
$k$ Möbius bands with $1\leq k\leq n-1$. One has $n\geq 3$
because $\alpha_1$ does not bound a Möbius band, so one can choose $\gamma_i$
contained in the $n$-punctured sphere.

If $\Sigma'\setminus \alpha_1$ is connected, 
consider $\beta$ a simple closed curve
intersecting $\alpha_1$ exactly once.
Then $\beta$ satisfies (1) or (2) according to whether
the regular neighbourhood $K$ of $\alpha_1\cup \beta$ is homeomorphic to
a once-punctured torus or to
a once-punctured Klein bottle. This concludes the proof of our claim.

If we have found $\beta\subset \Sigma'$ satisfying (1),
then by \cite{Stukow_Dehn} (see also
\cite[Th 11]{Paris_beginners}), $i(\beta, \sigma^n(\beta))$ is unbounded and the lemma is proved.  

Consider $\beta$ satisfying (2)
and let $K$ be the regular neighbourhood of $\alpha_1\cup \beta$. Thus
$K$ is homeomorphic to a once-punctured Klein bottle, and its fundamental group is a free group $\grp{\alpha_1,\beta}$.
The set of simple curves of $K$ having bounded intersection number
with $\alpha_1$ and $\beta$ is finite, but the Dehn twist 
along $\alpha_1$ sends $\beta$ to $\alpha_1^n\beta$ which 
lie in distinct conjugacy classes.
We conclude that the intersection number of $\beta$ with
$\sigma^n(\beta)$ is unbounded, which proves the lemma.
\end{proof}

\section{Classification of the underlying variety} \label{sec_underlying_variety}

Recall that $h_e\in \HOM$ is the $\Gamma$-homomorphism associated to the neutral element of the algebraic group $V$.

\begin{figure}[h!] 
\begin{center}
\includegraphics[scale=.8]{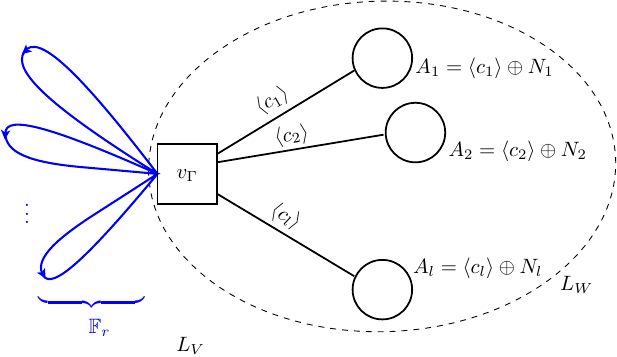}
\end{center}
\caption{The limit group $L_V=L_W*\bbF_r$ of a \conn\ algebraic group.}\label{fig_rose}
\end{figure}

\begin{thm} \label{thm_structure_V} Let $\Gamma$ be a torsion-free hyperbolic group.
Let $V$ be a \conn\ algebraic group over $\Gamma$ with coordinate group $L_V$.

Then 
there exist non-conjugate maximal cyclic groups $\grp{c_1}, \ldots, \grp{c_l}$ of  $\Gamma$ and $n_1, \ldots, n_l \in \bbN\setminus\{0\}$ such that 
$L_V$ is the free product of a free group with a multiple extension of centralizers $L_W$ of $\Gamma$ along the $c_j$'s as on Figure \ref{fig_rose},
where the abelian vertex groups are given by $A_j=\grp{c_j}\oplus N_j$, with
 $N_j=\ker h_{e|A_j} \simeq \bbZ^{n_j}$.

In particular, 
the algebraic variety $V$ is isomorphic to $\Gamma^r\times \grp{c_1}^{n_1}\times\dots\times \grp{c_l}^{n_l}$. 
\end{thm}

\begin{proof}
    Let $\Lambda$ be a standard JSJ decomposition of $L_V$.
    
    By Proposition \ref{prop_freeprod}, the Grushko decomposition of $L_V$ relative to $\Gamma$ has the form
    $L_V=L_V^0*\bbF_r$ with $\Gamma\subset L_{V}^0$. Thus, the subgraph $\Lambda_{0}\subset \Lambda$ spanned by vertices and edges with non-trivial vertex or edge group is connected, with $\pi_1(\Lambda_0)=L_{V}^0$.
    By Lemma \ref{lem_homogene}, the rigid group of $\Lambda_0$ containing $\Gamma$ is equal to $\Gamma$.
    We denote by $v_\Gamma$ the corresponding vertex of $\Lambda_0$.
    
    By Theorem \ref{thm_noQH} and Theorem \ref{thm_norigid}, $\Lambda_0$ has no vertex of surface type and no rigid vertex
    except $v_\Gamma$. Thus, $v_\Gamma$ is the only vertex of $\Lambda_0$ with non-abelian stabilizer. Moreover, since the Bass-Serre tree of $\Lambda_0$ is a tree of cylinders, it is bipartite between
    abelian and non-abelian vertices. It follows that every abelian vertex $v$ of $\Lambda_0$ is adjacent
    to $v_\Gamma$, and by Assertion \ref{it_boucle} of Lemma \ref{lem_homogene}, there is exactly one edge in $\Lambda_0$ joining $v$ to $v_\Gamma$. 
    Thus, $\Lambda_0$ is a star with center $v_\Gamma$. 
    The incident edge groups $\grp{c_j}\subset \Gamma$ are  
    not conjugate in $\Gamma$ by ayclindricity at $v_\Gamma$ ($\Lambda_0$ is a tree of cylinders).
    
    By Assertion \ref{it_maxZ} of Lemma \ref{lem_homogene}, 
    all edge groups of $\Lambda_0$ are maximal cyclic in $L_V$. 
    This shows that $L_V^0$ is a multiple extension of centralizers and
    Lemma \ref{lem_direct_sum} concludes the proof of the first part.

    The group $L_V$ has a presentation of the form
    $$L_V=\Big\langle \Gamma,x_1,\dots,x_r, y_{i,j} i\leq n_j, j\leq l
    \ \Big |\ [y_{i,j},y_{i',j}]=1, [y_{i,j},c_j]=1\Big\rangle.$$
    This identifies $V$ with the set of tuples in $\Gamma^{r+n_1+\dots +n_l}$ which satisfy
    the equations 
    $[y_{i,j},y_{i',j}]=1, [y_{i,j},c_j]=1$
    hence with $\Gamma^r\times \grp{c_1}^{n_1}\times\dots\times \grp{c_l}^{n_l}$.
\end{proof}

\section{Classification of the multiplication law: proof of Theorem \ref{thm_main_result}} \label{sec_mult_table}

The goal of this section is to prove Theorem \ref{thm_main_result}.

The first step of the proof was achieved by Theorem \ref{thm_structure_V}, 
which gives a full description of the underlying variety of an algebraic group $(V, \mu)$ over $\Gamma$. 
To complete the classification of algebraic groups, we need to classify the possible multiplication laws on $V$. Our goal is to find a system of coordinates 
in which $\mu$ is the standard multiplication law.

In subsection \ref{subsec_changing_coordinates}, we set up notations
for systems of coordinates on a general variety $V$, 
and the corresponding systems
on $V\times V$ and $V\times V\times V$. 
We use this to express associativity of the multiplication law when $V$ is an algebraic group.

In subsection \ref{subsec_coordinates_L_V}, we use the structure of the variety $V$ given by Theorem \ref{thm_structure_V} to
choose an adapted coordinate system.

In subsection \ref{sec_multiplication_abelian},
we show that in this system of coordinates,
the multiplication law is standard on the abelian part.

In subsection \ref{subsec_triangular}, we prove that we can change
our system of coordinates on $V$
such that, for each pair of elements $x,y\in V$,
multiplication on the left by $x$ and on the right by $y$
is a transformation $\Delta_{x,y}$ that is triangular in these coordinates.
There is nothing to do in the abelian part, and in the free product part,
we use the bounded stretch argument (Corollary \ref{cor_bound}) to 
exclude the possibility that $\Delta_{x,y}$ be fully irreducible. We then use \cite{HaMo_subgroup}\cite[Theorem 1]{GH_boundaries} to deduce that the action of $V\times V$ on $V$ by left and right multiplication is reducible. 
The triangular form follows by induction. 

This means that the $i$-th coordinate of 
$\Delta_{x,y}(v)$ depends only the first $i$ coordinates of $v$ (as well as on $x$ and $y$).
The dependence in $v$ is algebraic by definition, and
in subsection \ref{sec_formal}, we show that 
the dependence on $x$ and $y$ is also given by algebraic maps.
In subsection \ref{subsec_structure_of_mult}, we
make some computations based on the associativity
of the multiplication law to deduce a new coordinate system
in which the multiplication law becomes standard. 
This is the content of
Proposition \ref{prop_changement_var} which thus concludes
the proof of Theorem \ref{thm_main_result}.

\subsection{Systems of coordinates} \label{subsec_changing_coordinates}

Consider $V$ is a general variety. Starting with $T_1, \ldots, T_n\in L_V$ such that $L_V=\grp{\Gamma, T_1, \ldots, T_n}$,
we get an embedding of $V=\HOM$ into $\Gamma^n$ via $h\in \HOM\mapsto (h(T_1),\dots,h(T_n))$.
This way, $T_i$ becomes the $i$-th coordinate function on $V$, and we call $(T_1,\dots,T_n)$
a system of coordinates on $V$.

Such a system induces a system of coordinates $(X_1,\dots,X_n,Y_1,\dots, Y_n)$ for $V \times V$ as follows.
Denoting by $p_1,p_2:V\times V\ra V$ the two projections,
we define $X_i=T_i\circ p_1$ and $Y_i=T_i\circ p_2$.
Equivalently, $X_i=p_1^*(T_i)$ and $Y_i=p_2^*(T_i)$.
Viewing $L_{V\times V}$ as a quotient of $L_V*_\Gamma L_V$,
$X_i$ (resp. $Y_i$) 
is the image of $T_i$ under the left (resp. right) embedding of $L_V$ into $L_{V\times V}$.

Suppose now $V$ is an algebraic group with multiplication law $\mu: V \times V \to V$.
We can write $\mu$
in these coordinates as
$$\bMx{x_1\\\vdots\\ x_n}\fois\bMx{y_1\\\vdots\\ y_n} =
\bMx{\mu_1(\ul x,\ul y)\\\vdots\\ \mu_n(\ul x,\ul y)}.$$
The corresponding comultiplication law $\mu^*:L_V\ra L_{V\times V}$ is defined on the generating set 
by 
$$\mu^*:\left\{\Mx{
      T_1\mapsto \mu^*(T_1)=\mu_1(X_1,\dots,X_n,Y_1,\dots,Y_n)  \\
\vdots \\
T_n\mapsto \mu^*(T_n)=\mu_n(X_1,\dots,X_n,Y_1,\dots,Y_n)}
\right.$$

Our goal is to find coordinates $T_1, \ldots, T_n \in L_V$ in which $\mu$ is the standard multiplication law
i.e.,  $\mu^*(T_i)=X_iY_i$ for all $i\leq n$.

\begin{rem}\label{rem_example}
If we change the coordinate system $(T_1,\dots,T_n)$ on $V$, we need to change the coordinate system on $V^2$ accordingly.
For instance, if we change $T_i$ to $T'_i=T_i\m$, 
then $X_i$ changes to $X'_i=p_1^*(T'_i)=X_i\m$ and similarly, 
$Y_i$ changes to $Y'_i=Y_i\m$.
If we change $T_i$ to $T'_i=u T_i v$ for some $u,v\in L_V$, 
then $X_i$ changes to $X'_i=u^{(X)}X_i v^{(X)}$
where $u^{(X)}=p_1^*(u)\in L_{V^2}$ and $v^{(X)}=p_1^*(v)$ are the copies of $u$ and $v$ 
under the left embedding of $L_V$ into $L_{V^2}$; there are similar changes for $Y_i$.
\end{rem}

Just as for $V \times V$, the system of coordinates $(T_1, \ldots, T_n)$ on $V$ also induces a system of coordinates on $V \times V \times V$: we define $X_i,Z_i, Y_i\in L_{V^3}$ (in this order) 
as the image of $T_i$ under the left, middle, and right embedding of $L_V$
in $L_{V^3}$ respectively. 
This corresponds to coordinates associated to the embedding 
$V\times V\times V \subset \Gamma^n \times \Gamma^n \times \Gamma^n$.

Let $(\ul x,\ul z,\ul y)$ be a point of $V\times V\times V$.
Associativity then reads
$$(\ul x \fois \ul z)\fois \ul y= \ul x\fois (\ul z\fois \ul y)\quad\text{i.e.}\quad
\mu(\mu(\ul x,\ul z),\ul y)=\mu(\ul x,\mu(\ul z,\ul y))$$
$$\text{i.e.}\quad \mu \circ (\mu_{X,Z}\times \id_Y)=\mu \circ (\id_X\times \mu_{Z,Y})$$
where $\mu_{X,Z}\times \id_Y:V^3\ra V^2$ is the map sending $(\ul x,\ul z,\ul y)$ to $(\mu(\ul x,\ul z),\ul y)$ and
$\id_X\times \mu_{Z,Y}:V^3\ra V^2$ maps $(\ul x,\ul z,\ul y)$ to $(\ul x, \mu(\ul z,\ul y))$.

At the level of function groups, the maps induced by $\mu_{X,Z}\times \id_Y$ and $\id_X\times \mu_{Z,Y}$
are
$$(\mu_{X,Z} \times \id_Y)^*:
\left\{\begin{array}{rcl}
L_{V\times V}&\ra& L_{V\times V\times V}\\
X_i &\mapsto& \mu_i(\ul X,\ul Z) \\
Y_i &\mapsto& Y_i \\
\end{array}\right., \quad\text{and}\quad
(\id_X\times \mu_{Z,Y})^*:
\left\{\begin{array}{rcl}
L_{V\times V}&\ra& L_{V\times V\times V}\\
X_i &\mapsto& X_i \\
Y_i &\mapsto& \mu_i(\ul Z,\ul Y) \\
\end{array}\right.
$$
and associativity translates into
$$(\mu_{X,Z} \times \id_Y)^*\circ \mu^*=
(\id_X \times \mu_{Z,Y})^* \circ \mu^*.$$

\subsection{Coordinates on the groups $L_V$, $L_{V^2}$ and $L_{V^3}$} \label{subsec_coordinates_L_V}

We have seen in Theorem \ref{thm_structure_V} that the underlying
variety $V$ of an algebraic group is isomorphic to 
 $\Gamma^r\times \grp{c_1}^{n_1}\times\dots \times \grp{c_l}^{n_l}$ where $\grp{c_1},\dots,\grp{c_l}$ are non-conjugate maximal cyclic subgroups of $\Gamma$.
Moreover, the group $L_V$  can be written as the fundamental group of
a graph of groups $\Lambda$ as in Figure \ref{fig_rose}, 
\ie it is the free product of a free group $\bbF_r$ with a multiple extension of centralizers $\Gamma *_{\grp{c_j}} A_j$ with $A_j=\grp{c_j}\oplus N_j$ and $N_j=\ker h_{e|A_j}\simeq \bbZ^{n_j}$,
 where $h_e\in \HOM$ is the $\Gamma$-homomorphism
 corresponding to the neutral element $e\in V$.

We denote by $T_1,\dots, T_r$ a free basis of the free group $\bbF_r$, and for all $j\leq l$, we denote by
$T_{j,1},\dots, T_{j,n_j}$ a basis of the free abelian group $N_j$, so that in particular $h_e(T_{j,k})=1$ for all $j,k$
(see Figure \ref{GoG_V_times_V}).
We say that the $T_i$ ($i\leq r$) are the free generators, and the $T_{j,k}$ ($j\leq l$, $k\leq n_j$) are the abelian generators.

Let $p_1,p_2:V\times V\ra V$ the two projections.
As described above, from the coordinate system $T_i,T_{j,k}$ on $V$,
we get a coordinate system $X_i$, $X_{j,k}$, $Y_i$, $Y_{j,k}$ on $V\times V$ defined by
$X_i=p_1^*(T_i)$, $X_{j,k}=p_1^*(T_{j,k})$, $Y_i=p_2^*(T_i)$ and $Y_{j,k}=p_2^*(T_{j,k})$ (see Figure \ref{GoG_V_times_V}). 
We denote by $N_j^{(X)}=\grp{X_{j,1},\dots,X_{j,n_j}}=p_1^*(N_j)$ (resp. $N_j^{(Y)}=\grp{Y_{j,1},\dots,Y_{j,n_j}}=p_2^*(N_j)$)
the two copies of $N_j$ in $L_{V\times V}$.
Similarly, when dealing with $V^3$, we will use the coordinate system $X_i$, $X_{j,k}$, $Z_i$, $Z_{j,k}$, $Y_i$, $Y_{j,k}$ on $V^3$ associated to the three projections $V^3\ra V$. We also define $N^{(X)}_j$, $N^{(Z)}_j$ and $N^{(Y)}_j$ the corresponding copies of $N_j$ in $L_{V^3}$.

\begin{lem} \label{lem_GoG_V_times_V}
The function groups $L_{V^2}$ of $V\times V$ and $L_{V^3}$ of $V\times V\times V$
are both free products of multiple extensions of centralizers $\grp{c_1}, \ldots ,\grp{c_l}$ of $\Gamma$ with a free group, of rank $2r$ and $3r$ respectively (see Figure \ref{GoG_V_times_V}),
where for $j\leq l$, the abelian vertex group of $L_{V^2}$ containing $\grp{c_j}$ is 
$$A_j^{V^2}=\grp{c_j}\oplus \grp{X_{j,1},\dots,X_{j,n_j}}\oplus \grp{Y_{j,1},\dots,Y_{j,n_j}}$$
$$= \grp{c_j}\oplus N_j^{(X)}\oplus N_j^{(Y)}\simeq \grp{c_j}\oplus \bbZ^{2n_j}$$
and the abelian vertex group of $L_{V^3}$ containing $\grp{c_j}$ is 
$$A_j^{V^3}=\grp{c_j}\oplus
\grp{X_{j,1},\dots,X_{j,n_j}}\oplus \grp{Z_{j,1},\dots,Z_{j,n_j}}\oplus \grp{Y_{j,1},\dots,Y_{j,n_j}}$$
$$=\grp{c_j}\oplus N_j^{(X)}\oplus N_j^{(Z)}\oplus N_j^{(Y)}
\simeq \grp{c_j}\oplus \bbZ^{3n_j}.$$

Their respective JSJ decompositions 
are shown on Figure \ref{GoG_V_times_V}.
\end{lem}

\begin{figure}[ht!] 
\centering
\includegraphics[width=\linewidth]{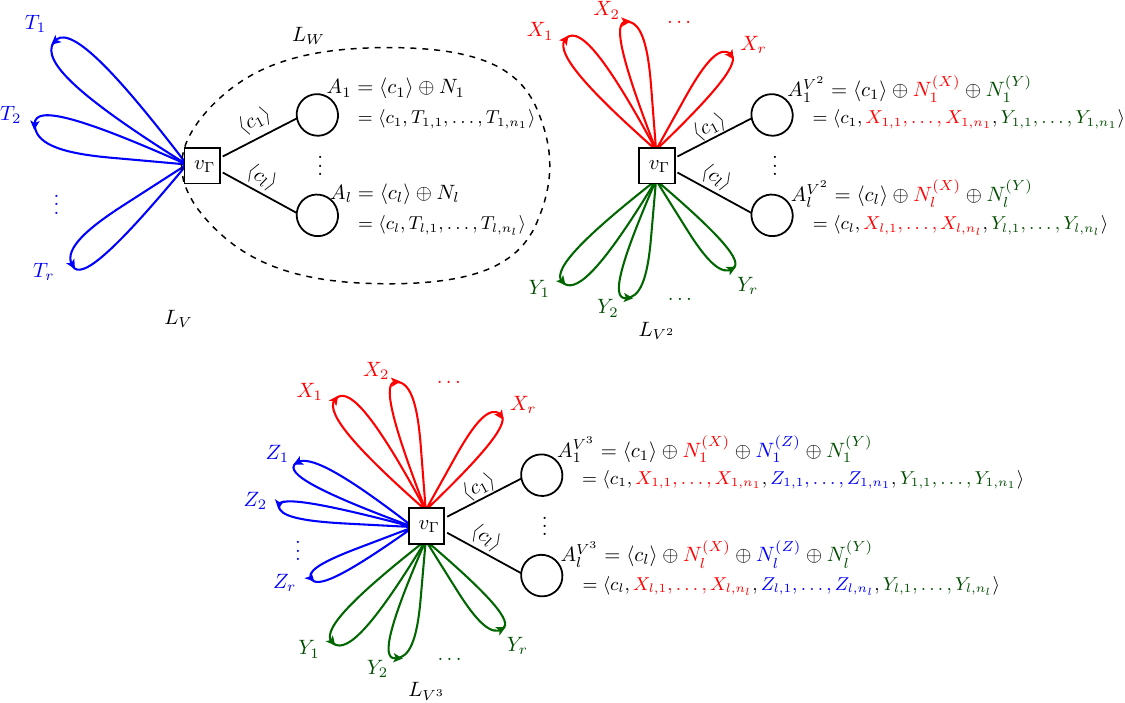}
\caption{The groups $L_V$, $L_{V^2}$, and $L_{V^3}$ with their generating sets, and the corresponding graphs of groups $\Lambda_V$, $\Lambda_{V^2}$ and $\Lambda_{V^3}$.} 
\label{GoG_V_times_V}
\end{figure}

\begin{proof}
By Lemma \ref{lem_direct_product}, $L_{V^2}$ is the maximal residually $\Gamma$ quotient of the double $D=L_V*_\Gamma \bar L_{V}$.
For any $\Gamma$-homomorphism $h\in \HOM[D]$, and any $j\leq l$, $h(N_j)$ and $h(\bar N_j)$ commute with $h(c_j)=h(\bar c_j)\neq 1$,
so commute together since $\Gamma$ is CSA. These commutation relations therefore hold in $L_{V^2}$.
Now the quotient of $D$ by these commutation relations is the fundamental group of the graph of groups
on Figure \ref{GoG_V_times_V}. Since the fundamental group of this graph of groups is a $\Gamma$-limit group, it is in particular residually $\Gamma$ and is thus equal to $L_{V^2}$.  
The case of $L_{V^3}$ is similar.
\end{proof}

\subsection{Multiplication law on the abelian part}
\label{sec_multiplication_abelian}

Recall that we defined a coordinate system $T_i, T_{j,k}\in L_V$ for which,
for all $j\leq l$, $$\ker h_{e|A_j}=N_j=\grp{T_{j,1}, \ldots, T_{j, n_j}}.$$
We now show that in these coordinates the abelian part of the multiplication law is standard.

\begin{prop}\label{prop_loi_abelienne}
With the bases chosen above, for all $j\leq l$, 
we have 
$$\mu^*(T_{j,k})=X_{j,k} Y_{j,k} \in L_{V^2}.$$
\end{prop}

\begin{proof} 
Fix $j\leq l$.
Since $T_{j,k}$ commutes with $c_j$ in $L_V$, 
$\mu^*(T_{j,k})$ commutes with $c_j$ in $L_{V^2}$, and we can write $\mu^*(T_{j,k})= c_j^{q_{j,k}} u_{j,k}v_{j,k}$ where $u_{j,k}\in N_j^{(X)}$ and $v_{j,k} \in N_j^{(Y)}$.

Denote by $i_e,i_e': V \to V^2$ the maps defined by $v \mapsto (e,v)$ and $v \mapsto (v,e)$.
At the level of function groups, $i_e^*:L_{V^2}\ra L_V$
is induced by the map $L_V*_\Gamma L_V\ra L_V$ which is $h_e$ on the left factor
and the identity on the right factor (sending $Y_i$ and $Y_{j,k}$ to $T_i$ and $T_{j,k}$).
In particular, $N_j^{(X)}\subset \ker i_e^*$.
Similarly, $i_e'^*(X_i)=T_i$, $i_e'^*(X_{j,k})=T_{j,k}$,
and $i_e'^*$ coincides with $h_e$ on the right factor, so $N_j^{(Y)}\subset \ker i_e'^*$.

Since $\mu \circ i_e = \mu \circ i'_e = \id_V$, we get that $i_e^* \circ \mu^* = i_e'^* \circ \mu^* = \id_{L_V}$.
Using that $u_{j,k}\in N_j^{(X)}$ lies in $\ker i^*_e$ we get
$$T_{j,k}=i_e^*(\mu^*(T_{j,k})) = i_e^*(c_j^{q_{j,k}} u_{j,k}v_{j,k})= c_j^{q_{j,k}} i_e^*(v_{j,k})$$
so $q_{j,k}=0$ and $i_e^*(v_{j,k})=T_{j,k}$ hence $v_{j,k}=Y_{j,k}$.

For similar reasons, $u_{j,k}= X_{j,k}$. 
We have proved that $\mu^*(T_{j,k})=X_{j,k}Y_{j,k}$ as desired.

\end{proof}

\subsection{A triangular representation}
\label{subsec_triangular}

For $x,y \in V$, recall that we denote $\Delta_{x,y}: V \to V; v \mapsto x \fois v \fois y$. We consider the dual automorphism $\Delta^*_{x,y}\in\Aut_\Gamma( L_V)$. 
The map $(x,y)\in V\times V\mapsto \Delta^*_{\inv x,y}$ 
is a representation of $V\times V$ in $\Aut_\Gamma(L_V)$ 
(noting that we have used the inverse $\inv x$ as the left multiplier).
The following result says that, up to conjugacy, this representation has triangular image. 

\begin{thm}\label{thm_triangulaire}
 Let $\Gamma$ be a torsion-free hyperbolic group.
Let $V$ be a \conn\ algebraic group over $\Gamma$ with function
group $L_V=L_{W}*\bbF_r$
where $L_W$ is the Grushko factor of $L_V$ relative to $\Gamma$
containing $\Gamma$.

Then there exist some free factors $L_V^r=L_V\supset L_V^{r-1}\supset \dots \supset L_V^0=L_W$
and a finite index subgroup $V_0\subset V$
such that for all $i\leq r$, $L_V^{i}\simeq L_V^{i-1} * \bbZ$
and $L_V^i$ is invariant under $\Delta_{x,y}^*$ for all $x,y\in V_0$.
\end{thm}

We will rely on the following result.

\begin{thm}[{\cite{HaMo_subgroup} \cite[Theorem 1]{GH_boundaries}}]\label{thm_iwip}
Let $H\subset \Out(G_1*\dots*G_k*\bbF_r)$ be a subgroup that preserves the conjugacy class of each $G_i$.
Let $\calf$ be a maximal $H$-periodic free factor system relative to $\{G_1,\dots, G_k\}$.

If $\calf$ is non-sporadic, then $H$ contains an element $\phi$ preserving $\calf$ and which is fully irreducible relative to $\calf$.
\end{thm}

If $G$ is a free product $G=G_1 * \ldots *G_k*\F_r$, a free factor system for $G$ relative to $\{G_1, \ldots, G_k\}$ 
is a collection $\calf = \{[U_1], \ldots, [U_p]\}$ of conjugacy classes of non trivial proper subgroups $U_1, \ldots, U_p$ of $G$ for which $G=U_1*\ldots*U_p*\F_l$, and such that each $G_i$ is contained in a conjugate of some $U_j$. 
Such a system is said to be sporadic if it corresponds to a decomposition of the form $G=U_1*U_2$ or $G=U_1*\F_1=U_1*\Z$. 
We say that this free factor system is $H$-periodic if some finite index subgroup of $H$ preserves the conjugacy class of each $G_i$.

We compare free factor systems as follows: $\{[U_1], \ldots, [U_p]\}\preceq \{[U'_1], \ldots, [U'_{p'}]\}$
if for all $i\leq p$, there exists $j\leq p'$ such that $U_i$ is conjugate into $U'_j$.
For example, a free factor system $\calf$ of $G$ is relative to $\{[G_1], \ldots, [G_k]\}$  if $\{[G_1], \ldots, [G_k]\}\preceq\calf$.

\begin{proof}[Proof of Theorem \ref{thm_triangulaire}]
    The natural map $\Aut_\Gamma(L_V)\ra\Out(L_V)$ is injective (no nontrivial inner automorphism restricts to the identity on $\Gamma$), so we view $\Aut_\Gamma(L_V)$ as a subgroup of $\Out(L_V)$.

    Consider the subgroup $H=\{\Delta_{x,y}^* \mid x,y\in V\}$ of $\Aut_\Gamma(L_V)$.
    The group $H$ preserves $L_W$, so Theorem \ref{thm_iwip} applies
    to $H\subset \Out(L_W*\bbF_r)$.
    Let $\calf$ be a maximal $H$-periodic free factor system of $L_V$ relative to $L_W$. 
    We denote by $H_0$ the finite index subgroup of $H$ which preserves $\calf$. 
    
    We claim that $\calf$ is sporadic. Otherwise, by Theorem \ref{thm_iwip}, there exists $\phi=\Delta^*_{u,v} \in H_0$ which is fully irreducible relative to $\calf$.
    
    Consider a train-track representative $f:T\ra T$ of $\phi$ where $T$ is a tree corresponding to some free splitting of $L_V$ relative to $\calf$ \cite{FrMa_stretching}, $\lambda>1$ its scaling factor. 
    Let $g\in L_V$ be a non-$\calf$-peripheral element whose axis is $f$-legal,
    i.e. such that for all $n\geq 1$, $f^n$ is injective in restriction to the axis of $g$ 
    (this restriction is therefore a homothety of factor $\lambda^n$ to the axis of $\phi^n(g)$).
    Then for all $n\in\bbN$, one has $||\phi^n (g)||_T=\lambda^n||g||_T$.
    On the other hand, one has $\phi^n(g)=\Delta_{u^{\fois n},v^{\fois n}}^*(g)$. 
    Applying Corollary \ref{cor_bound2} to the space $Y=T$,
    we get that $||\Delta_{u^{\fois n},v^{\fois n}}^*(g)||_T$ is bounded by a constant $C_g$ which does not depend of $n$. This contradiction shows that $\calf$ is sporadic, so that either $L_V=A*B$ with $\calf=\{[A],[B]\}$ 
    or $L_V=A*\bbZ$ with $\calf=\{[A]\}$ with $\Gamma\subset A$ in both cases.
    
    We now show that the former cannot hold. 
    Consider a morphism $h \in \HOM$ which kills $B$. The set $\{ h \circ \alpha | \alpha \in H \}$ is discriminating by Lemma \ref{lem_discriminating_rho} since $\Delta^*_{1,v}=\rho^*_v$ lies in $H$ for all $v \in V$. 
    By Lemma \ref{lem_indice_fini}, the family $\{ h \circ \alpha | \alpha \in H_0 \}$ is still discriminating. Since all these morphisms kill $B$, this is a contradiction.
    
    Therefore, one can write $L_V=L_V^{r-1}*\bbZ$ with $\calf=\{[L_V^{r-1}]\}$ and $L_W\subset L_V^{r-1}$. 
    We let $V_{r-1}$ be a finite index subgroup of $V$ such that $\Delta^*_{x,y}$ preserves $L_V^{r-1}$ up to conjugacy for all $x,y \in V_{r-1}$.
    The subgroup $V_{r-1}$ can be built as follows: consider the group morphism $\Delta:V \times V \to H$ defined by $(x,y) \mapsto \Delta^*_{u^{-1},v}$; then $\Delta\m(H_0)\subset V\times V$  projects to finite index subgroups 
    of each of the two factors: take $V_{r-1}$ to be their intersection.
    
    Since $\Delta^*_{x,y}$ is the identity on $\Gamma$, we note that $L_V^{r-1}$ is actually invariant (not only up to conjugacy).
    
    We now iterate the argument, some care is needed because it is not a priori clear that the limit group $L_V^{r-1}$ corresponds to an algebraic group. 
    Assume that we have constructed $L_V=L_V^r\supset L_V^{r-1}\supset\dots\supset L_V^{i}$
    where for all $j\in \{i+1,\dots, r\}$, $L_V^{j}=L_V^{j-1}*\bbZ$, and for all $j\in \{i,\dots, r\}$, $L_V^j\supset L_W$ is invariant under $\Delta^*_{x,y}$ for all $x,y$ in a finite index subgroup $V_{j}\subset V$. 
    Assume $L_V^i\neq L_W$.
    
    Consider 
    $H' = \{(\Delta^*_{x,y})_{|L_V^i} \mid x,y \in V_i\}$. 
    Let $\calf'$ be a maximal $H'$-periodic free factor system of $L_V^i$ relative to $L_W$. As above, we claim that $\calf'$ is sporadic.
    If not, $H'$ contains an automorphism $\phi\in\Aut_\Gamma(L_V^i)$ which is fully irreducible relative to $\calf'$. 
    From $T'$ a $L_V^i$-tree carrying a train track representative $f$ of $\phi$, one can find an element $g\in A_i$ whose axis is $f$-legal, thus  
     $||\phi^n(g)||_{T'}\ra \infty$ as $n$ goes to infinity.
    
    One can extend the action of $L_V^i$ on $T'$ to an action of $L_V$ on a larger tree $Y$:
    consider the graph of groups $T'/L_V^i$ and add $r-i$ new loops with trivial edge group
    attached to the vertex whose vertex group contains $L_W$; one can take for $Y$ the Bass-Serre tree of the obtained graph of groups.
    By Corollary \ref{cor_bound2}, $||\phi^n(g)||_Y$ is bounded.
    Since $||\phi^n(g)||_Y=||\phi^n(g)||_{T'}$, we get a contradiction
    showing that $\calf'$ is sporadic.

    Arguing as above, $\calf'$ has to consist in a single factor 
    $\calf'=\{[L_V^{i-1}]\}$ with $L_V^i=L_V^{i-1}*\bbZ$ and $L_V^{i-1}\supset L_W$ is invariant by $\Delta_{x,y}^*$ 
    for $x,y$ in some finite index subgroup $V_{i-1}$ of $V$.

    Iterating the argument at most $r$ times, we get the result.
\end{proof}

By Theorem \ref{thm_triangulaire}, one can write $L_V$ as a free product $L_V=L_W*\grp{T_1,\dots,T_r}$ where $\grp{T_1,\dots,T_r}$ is free of rank $r$,
and such that for all $i\leq r$, the subgroup $L_V^i=L_W*\grp{T_1,\dots,T_i}$ 
is invariant under $\Delta^*_{x,y}$ for all $x,y$ in the finite index subgroup $V_0\subset V$. 

Since any automorphism of $L^{i-1}_V*\grp{T_i}$ preserving $L^{i-1}_V$
must send $T_i$ to $g T^{\pm 1}_i g'$ for some $g,g'\in L^{i-1}_V$,
Theorem \ref{thm_triangulaire} can be reformulated as follows:

\begin{cor}\label{cor_triangulaire}
Let $T_1,\dots,T_r\in L_V$ 
be as above. 

Then for all $i\leq r$, $x,y\in V_0$, the map $\Delta_{x,y}^*:L_V\ra L_V$ acts as follows
     $$\Delta_{x,y}^*:T_i \ \mapsto \ g_{i,x,y} T_i^{\pm 1}g'_{i,x,y}$$
     where $g_{i,x,y}, g'_{i,x,y}\in L_V^{i-1}=\grp{L_W,T_1,\dots,T_{i-1}}$.
\end{cor}

\subsection{Formal triangularization}\label{sec_formal}

Corollary \ref{cor_triangulaire} says that in an appropriate coordinate system
$\Delta_{x,y}^*$ maps $T_i$ to $g_{i,x,y} T_i^{\pm 1}g'_{i,x,y}$,
where $g_{i,x,y}$ and $g'_{i,x,y}$
are elements which can be written as $\Gamma$-words on $T_1,\dots,T_{i-1}$ and on the abelian generators $T_{j,k}$;
but it does not say anything about the dependence of $g_{i,x,y}$ and $g'_{i,x,y}$
with respect to $x$ and $y$.
Using the algebraic group structure, we are going to prove that
this dependence is algebraic, i.e. that $g_{i,x,y}$ and $g'_{i, x, y}$ are given by $\Gamma$-words in the coordinates of $x$, the coordinates of $y$, the generators $T_1, \ldots T_{i-1}$ and the abelian generators $T_{j,k}$,
see Proposition \ref{prop_M_etoile} for the precise statement.

Let us now introduce some notation to state our result properly.

We denote by $T_i, T_{j,k}\in L_V$ the coordinate system given by Proposition \ref{prop_loi_abelienne} and Corollary \ref{cor_triangulaire}.
In particular $L_V = L_W * \grp{T_1, \ldots T_r \mid \; }$ where 
$L_W$ is generated by $\Gamma$ and the abelian generators $T_{j,k}$
(see Figure \ref{GoG_V_times_V}).
We denote by $X_i$, $X_{j,k}$, $Y_i$, $Y_{j,k}$
(resp.\ $X_i$, $X_{j,k}$, $Z_i$, $Z_{j,k}$, $Y_i$, $Y_{j,k}$)
the corresponding coordinate system on $L_{V^2}$ (resp.\ on $L_{V^3}$).

We denote by 
$L_{V}^{(X)}=\grp{\Gamma,X_{i},X_{j,k}|\; j\leq l, k\leq n_j,i\leq r}$
\newcommand{\LV}[2]{L_V^{(#1,#2)}}
the copy of $L_V$ in $L_{V^2}$ or $L_{V^3}$ (according to the context)
generated by the $X$ variables (in red on Figure \ref{GoG_V_times_V}),
and we define $L_{V}^{(Z)}$ and $L_{V}^{(Y)}$ similarly 
(where the $Z$ version only makes sense in $L_{V^3}$).

We then define the corresponding subgroups
$L_{W}^{(X)}=\grp{\Gamma,X_{j,k}|\; j\leq l,k\leq n_j}\subset L_V^{(X)}$,
and $\LV{i}{X}=\grp{L_W^{(X)},X_1,\dots,X_i}\subset L_V^{(X)}$ for $i\leq r$
the copies of $L_W$ and $L_V^i$ in $L_{V}^{(X)}$.
We define $L_{W}^{(Z)}$, $\LV{i}{Z}$,
$L_{W}^{(Y)}$, and $\LV{i}{Y}$ similarly.
Thus, the subgroup $\grp{L_{V}^{(X)},\LV{i-1}{Z},L^{(Y)}_{V}}\subset L_{V^3}$
appearing in the next statement is the set of elements of $L_{V^3}$
that do not involve the generators $Z_{i},\dots,Z_{r}$.

\begin{prop} \label{prop_M_etoile}
    The map $M^*:L_{V}\ra L_{V^3}$ sends
    $T_i$ to $g_i Z_{i}{}^{\pm 1} g'_i$
    where the elements $g_i,g'_i\in L_{V^3}$ lie in $\grp{L_{V}^{(X)},\LV{i-1}{Z},L^{(Y)}_{V}}\subset L_{V^3}$.
\end{prop}

As mentioned above, this result says that 
the elements $g_{i,x,y}$ and $g'_{i,x,y}$ given in Corollary \ref{cor_triangulaire}
depend algebraically 
on $x$ and $y$ (they are now represented by fixed words on $X$ and $Y$).
Additionally, the restriction $x,y\in V_0$ appearing in Corollary \ref{cor_triangulaire} has vanished.

Given $x, y \in V$ we denote $i_{x,y}: V \to V^3$ the map defined by $v \mapsto (x, v, y)$. 
Its dual map $i^*_{x, y}: L_{V^3} \to L_V$ can be described as follows: 
if $h_{x},h_y\in\HOM$ denote the $\Gamma$-homomorphisms corresponding to $x$ and $y$,
the restrictions of $i^*_{x,y}$ on $L_V^{(X)}$ and $L_V^{(Y)}$ coincides with $h_{x}$ and $h_y$
(with the natural identification)
 and its restriction to $L^{(Z)}_{V}$ is the identity (sending $Z_i$ to $T_i$ and $Z_{j,k}$ to $T_{j,k}$). 
Note for future use that this implies that $i^*_{x,y}$ maps the group $U=\grp{L^{(X)}_{V},L^{(Z)}_{W},L^{(Y)}_{V}}$ in $L_W$.

To prove Proposition \ref{prop_M_etoile} we will need the following lemma. Recall that $V_0$ is the finite index subgroup of $V$ from Theorem \ref{thm_triangulaire}: for any $x,y\in V_0$, the map $\Delta_{x,y}^*$ is triangular.

\begin{lem} \label{lem_collection}
       The collection of  $\Gamma$-homomorphisms $i^*_{x,y}: L_{V^3} \to L_V$
       with $(x,y)\in V_0^2$ is a discriminating family.
\end{lem}

\begin{proof}
    We view $L_{V^3}$ and $L_V$ as groups of functions on $V^3$ and $V$ with values in $\Gamma$.
    
    Given $f \in L_{V^3}$ we have that $i^*_{ x, y} (f)$ is the map $ f( x, \cdot, y)$ so its image by $h_v$ is $f(x, v, y)$. This shows that $h_v \circ i^*_{x,y} = h_{(x,v,y)}$. 
    
    Since $V_0\times V\times V_0$ is a finite index subgroup of the \conn\ algebraic group $V^3$,
    this is a dense subset of $V^3$ (Lemma \ref{lem_fidx_dense}) and by Corollary \ref{lem_discriminating_V0},
    the family of $\Gamma$-homomorphisms 
    $$\{h_{(x,v,y)} |\, (x,v,y)\in V_0\times V\times V_0\}$$ is discriminating.
    Since $h_{(x,v,y)}$ factors through $i^*_{x,y}$ this concludes the proof.
\end{proof}

\begin{proof}[Proof of Proposition \ref{prop_M_etoile}]
    Denote $U=\grp{L^{(X)}_{V},L^{(Z)}_{W},L^{(Y)}_{V}}\subset L_{V^3}$.
    Write $L_{V^3}$ as the free product $U*\bbF_r$ where $\bbF_r$ is the free group $\bbF_r=\grp{Z_1,\dots,Z_{r}}$.
    Write $M^*(T_i)$ as a reduced word in this free product:
    $M^*(T_i)=u_0 w_1 u_1 w_2\dots u_n w_{n+1}$ with $u_j \in U, w_j \in \bbF_r$ and $u_j, w_j$ non-trivial except maybe for $u_0, w_{n+1}$.
    It suffices to prove that $w_j\in \grp{Z_1,\dots,Z_i}$ for all $j$, and
    that  $Z_i^{\pm 1}$ appears exactly once in exactly one of the $w_j$'s.
    
    By Lemma \ref{lem_collection}, there exist $x,y \in V_0^2$ such that for all $j$, the element $u'_j=i^*_{(x,y)}(u_j)\in L_V$ is non-trivial whenever $u_j$ is non-trivial. 
    As noticed above, $i_{x,y}^*(U)\subset L_W$ so $u'_j\in L_W$.
    Now 
    $$ i^*_{(x,y)} \circ M^* (T_i) = u'_0 w_1 u'_2 w_2\dots u'_n w_{n+1}\in L_W * \grp{T_1,\dots,T_r}$$
    (under the natural identification of $\grp{Z_1,\dots,Z_r}\subset L^{(Z)}_{V}$ with its copy $\grp{T_1,\dots,T_r}\subset L_V$). 
    This is a reduced word in the free product $L_V = L_W * \grp{T_1,\dots,T_r}$.
    
    On the other hand, we have 
    $\Delta_{x,y} = M \circ i_{(x,y)}$  so $\Delta^*_{x,y} =  i^*_{(x,y)} \circ M^*$
    so by Corollary \ref{cor_triangulaire}, 
    $$i^*_{(x,y)} \circ M^* (T_i) = g_{i,x,y}T_i^{\pm 1}g'_{i,x,y}$$
    where $g_{i,x,y},g'_{i,x,y} \in L_V^{i-1}=L_W*\grp{T_1,\dots,T_{i-1}}$. 
    This shows that $w_1, \ldots, w_{n+1}$ lie in $\grp{T_1, \ldots, T_i}$ (\ie the generators $T_{i+1},\dots T_r$ do not appear) and that $T_i^{\pm1}$ appears exactly once in exactly one of the $w_j$.
    
    This proves the proposition.
\end{proof}

\subsection{Structure of the multiplication law}
\label{subsec_structure_of_mult}

Recall that $\mu:V\times V\ra V$ is the multiplication law, and $\mu^*:L_V\ra L_{V^2}$
is the dual map.

We have already seen that 
the  multiplication is standard
in restriction to the abelian generators of $L_W\subset L_V$ 
(see Proposition \ref{prop_loi_abelienne}).

We now turn to the description of $\mu^*$ on the free generators $T_1, \ldots, T_r$ of $L_V$.

We want to show that up to changing the coordinate system, $\mu^*: T_i \mapsto X_iY_i$. We do this in several steps. Proposition \ref{prop_mu_etoile} shows that in the coordinate system given by the previous sections, $X_{i+1},Y_{i+1},\dots,X_r, Y_{r}$ do not appear at all in $\mu^*(T_i)$,
and $X_i$ and $Y_i$ each appear exactly once (possibly with exponent $-1$). 
In Lemma \ref{lem_positif}, we see that we can arrange that they appear with exponent $+1$ and in the correct order. Proposition \ref{prop_changement_var} gives a last change of coordinate system in which the product is standard.

Denote by $L_{W^2} \subset L_{V^2}$ (resp. $L_{W^3}\subset L_{V^3}$) the subgroup 
generated by $\Gamma$ and all abelian generators.
For $i\leq r$, we denote by $L_{V^3}^{i}=\grp{L_{W^3},X_1,Z_1,Y_1,\dots,X_i,Z_i,Y_i}$
and $L_{V^2}^{i}=\grp{L_{W^2},X_1,Y_1\dots,X_i,Y_i}$, \ie an element lies in $L_{V^3}^{i}$ (resp. $L_{V^2}^{i}$) if it does not involve the free generators with index greater than $i$.

\begin{prop} \label{prop_mu_etoile}
For all $i\leq r$,
there exists $a_i,b_i,c_i\in L_{V^2}^{i-1}$ 
such that either
    $$\mu^*(T_i)= a_i\cdot \left( X_i \right)^{\pm 1} \cdot b_i\cdot \left( Y_i \right)^{\pm 1} \cdot c_i$$
    or
    $$\mu^*(T_i)= a_i \cdot \left(Y_i \right)^{\pm 1}  \cdot b_i \cdot \left( X_i \right)^{\pm 1} \cdot c_i.$$
\end{prop}

\begin{proof}
    Since $\mu(x,y)=M(x,y,e)=M(e,x,y)$,
    one has $\mu^*=i_e^*\circ M^*=i_e'^*\circ M^*$
    where $i_e,i'_e:V^2\ra V^3$ are defined by $i_e(x,y)=(x,y,e)$ and $i_e'(x,y)=(e,x,y)$.
    The dual morphisms $i_e^*$ and $i_e'^*$ satisfy
  \begin{align*}
    i_e^*:L_{V^3}&\ra L_{V^2} &i_e'^*:L_{V^3}&\ra L_{V^2} \\
    X_i,X_{j,k} &\mapsto X_i,X_{j,k} & X_i,X_{j,k} &\mapsto h_e(X_i),h_e(X_{j,k})=1 \in\Gamma\\  
    Z_i,Z_{j,k} &\mapsto Y_i,Y_{j,k} & Z_i,Z_{j,k} &\mapsto  X_i,X_{j,k}\\  
    Y_i,Y_{j,k} &\mapsto h_e(Y_i),h_e(Y_{j,k})=1\in \Gamma  & Y_i,Y_{j,k}&\mapsto Y_i,Y_{j,k}
  \end{align*}
Fix $i\leq r$ and recall that $L_{V^2}=L_{V^2}^{i-1}*\grp{X_i,Y_i,\dots,X_r,Y_r}$.
Write $\mu^*(T_i)$ as a reduced word as follows: 
\begin{equation}\label{eq1}
    \mu^*(T_i)=g_0 S_1 g_1 \dots S_n g_n
\end{equation}
where for each $j\leq n$, $S_j$ is a letter in $\{X_i^{\pm 1},Y_i^{\pm 1},\dots,X_r^{\pm 1},Y_r^{\pm 1}\}$, 
and $g_j\in L_{V^2}^{i-1}$, and $g_j$ is nontrivial 
when $S_{j+1}=S_j\m$.

By Proposition \ref{prop_M_etoile} we have 
$$M^*(T_i)= g_i {Z_{i}}^{\pm 1}g'_i \text{ for some }g_i,g'_i\in\grp{L_{V^3}^{i-1},X_i,Y_i\dots, X_r,Y_r}.$$ 
Thus
$$\mu^*(T_i)=i^*_e \circ M^*(T_i) =i_e^*(g_i){Y_i}^{\pm 1}  i_e^*(g'_i)$$
with $i_e^*(g_i),i_e^*(g'_i)\in \grp{L_{V^2}^{i-1},X_i,\dots, X_r}$.
This shows that the letters $Y_{i+1}^{\pm 1},\dots, Y_r^{\pm 1}$ don't occur in (\ref{eq1}), 
and that $Y_i^{\pm 1}$ occurs exactly once.

By a similar argument using $i_e'^*$, we see that the letters $X_{i+1}^{\pm 1},\dots, X_r^{\pm 1}$ don't occur in (\ref{eq1}), 
and that $X_i^{\pm 1}$ occurs exactly once.
This concludes the proof.
\end{proof}

The following lemma says that we may assume that the exponents of $X_i ,Y_i$ 
in Proposition \ref{prop_mu_etoile} are positive and that $X_i$ comes before $Y_i$.
\begin{lem}\label{lem_positif}
       Up to changing $T_i$ to $T_i\m$ and $X_i,Y_i$ to $X_i\m,Y_i\m$ accordingly (see Remark \ref{rem_example}), 
       there exists $a_i,b_i,c_i\in L_{V^2}^{i-1}$ such that 
       $$\mu^*(T_i)=a_i X_i b_i Y_i  c_i.$$
\end{lem}

\begin{proof}
We apply Proposition \ref{prop_mu_etoile}.
If $\mu^*(T_i)= a_i  Y_i ^{\eps} b_i X_i^{\eps'} c_i$, 
for some $\eps,\eps'\in \{\pm 1\}$ and $a_i,b_i,c_i\in L_{V^2}^{i-1}$,
then we change $T_i$ to $T'_i=T_i\m$
and $X_i,Y_i$ to $X'_i=X_i\m, Y'_i=Y_i\m$ accordingly.
This reverses the order of appearance of $X_i$ and $Y_i$:
we get $\mu^*(T'_i)= c_i\m   X_i'{}^{\eps'} b_i\m   Y_i'^{\eps} a_i\m$.
Thus, we may assume that
$$\mu^*(T_i)= a_i X_i ^{ \eps} b_i Y_i ^{ \eps'} c_i.$$

We use the fact that $\mu(v,e)=v$ for all $v\in V$, \ie 
$\mu\circ i_e=\id_V$ where $i_e:V\ra V^2$ is defined by $i_e(v)=(v,e)$.
This translates into $i_e^*\circ \mu^*=\id_{L_V}$.
Similarly, we have $i_e'^*\circ \mu^*=\id_{L_V}$
where $i'_e(v)=(e,v)$.
Note that 
\begin{align*}
    i_e^*:L_{V^2}&\ra L_V & i_e'^*:L_{V^2}&\ra L_V \\
    X_i,X_{j,k}&\mapsto T_i,T_{j,k} & X_i,X_{j,k} &\mapsto h_e(X_i),h_e(X_{j,k})=1\in  \Gamma \\
    Y_i,Y_{j,k}&\mapsto h_e(Y_i),h_e(Y_{j,k})=1 \in \Gamma &  Y_i,Y_{j,k}&\mapsto  T_i,T_{j,k}
\end{align*}

We have that  $i_e^*\circ \mu^*(T_i) = i_e^*(a_i) \cdot T_{i} ^{\eps} \cdot i_e^*(b_i Y_i ^{ \eps'} c_i)$. 
Since this must be equal to $T_i$ and since $i_e^*(a_i), i_e^*(b_i Y_i ^{\eps'} c_i) \in L_V^{i-1}$, the exponent
$\eps$
has to be positive.
Arguing with $i_e'^*$ instead of $i_e^*$, we similarly get $\eps'=+1$ which concludes the proof.
\end{proof}

From now on, we take $a_i,b_i,c_i\in L_{V^2}^{i-1}$ as in Lemma \ref{lem_positif}.
The following proposition is the last step in the proof of the main result.
$u^{(Y)}\in L_{V}^{(Y)}\subset L_{V^2}$ the images of $u$
under the embeddings of $L_V$ into $L_{V^2}$.

\begin{prop} \label{prop_changement_var} For each $i$ there exist $u_i, v_i \in L_V^{i-1}$ such that if we define $T'_i = u_iT_i v_i$, and change accordingly $X_i,Y_i$ to $X'_i=u_i^{(X)}X_iv_i^{(X)}$ and $Y'_i=u_i^{(Y)}Y_iv_i^{(Y)}$ (see Remark \ref{rem_example}),
we have $$\mu^*(T'_i) = X_i'  Y_i' .$$
\end{prop}

We will need two lemmas before proving the proposition.

Given $u\in L_{V^2}$, we denote by  $u^{(XZ)}\in \grp{L_{V}^{(X)},L_{V}^{(Z)}}\subset L_{V^3}$
the image of $u$
under the embeddings of $L_{V^2}$ into $L_{V^3}$
sending respectively $X_i,X_{j,k},Y_i,Y_{j,k}$ to $X_i,X_{j,k},Z_i,Z_{j,k}$.
Similarly we denote by $u^{(ZY)}\in \grp{L_{V^3}^{(Z)},L_{V^3}^{(Y)}}$ its image under the embedding sending respectively $X_i,X_{j,k},Y_i,Y_{j,k}$ to $Z_i,Z_{j,k},Y_i,Y_{j,k}$.

Define $\theta_L,\theta_R:V^3\ra V^2$ by $\theta_L(x,y,z)=(x\fois y,z)$
and $\theta_R(x,y,z)=(x,y\fois z)$. The dual maps satisfy
\begin{align*}
    \theta_L^*:L_{V^2}&\ra L_{V^3} & \theta_R^*:L_{V^2}&\ra L_{V^3}\\
    X_i&\mapsto a_i^{(XZ)} X_i b_i^{(XZ)} Z_i c_i^{(XZ)} & X_i &\mapsto X_i\\
    Y_i &\mapsto Y_i &   Y_i&\mapsto a_i^{(ZY)} Z_i b_i^{(ZY)} Y_i c_i^{(ZY)}.
\end{align*}
In particular, 
we have that $\theta_L^*(L_{V^2}^{i-1})\subset L_{V^3}^{i-1}$, and $\theta_R^*(L_{V^2}^{i-1})\subset L_{V^3}^{i-1}$.

Associativity says that $\mu\circ \theta_L=\mu\circ \theta_R$, \ie
$\theta_L^*\circ\mu^*=\theta_R^*\circ\mu^*$.
This is the basis of the following relations.

\begin{lem}\label{lem_equations}  With $a_i,b_i,c_i$ as in Lemma \ref{lem_positif} we have
in $L_{V^3}$
\begin{align}
    \theta_R^*(a_i) &= \theta_L^*(a_i)a^{(XZ)}_{i} \label{eq_a} \\
    b^{(XZ)}_{i} &= \theta_R^*(b_i) a^{(ZY)}_{i} \label{eq_b}\\
    b^{(ZY)}_{i} &= c^{(XZ)}_{i} \theta_L^*(b_i) \label{eq_c}\\
    \theta_L^*(c_i) &= c^{(ZY)}_{i} \theta_R^*(c_i) \label{eq_d}
\end{align}
\end{lem}

\begin{proof} Recall that we have $\mu^*(T_i)=a_iX_i b_i Y_i c_i\in L_{V^2}$, where $a_i,b_i,c_i\in L_{V^2}^{i-1}$ 
(\ie do not involve $X_{i},Y_i,\dots, X_r,Y_r$).
Thus 
$$\theta_L^*(\mu^*(T_i))=\theta_L^*(a_i) \cdot a_i^{(XZ)} X_i  b_i^{(XZ)} Z_i c_i^{(XZ)} \cdot \theta_L^*(b_i) \cdot Y_i \cdot \theta_L^*(c_i)$$
and
$$\theta_R^*(\mu^*(T_i))=\theta_R^*(a_i)\cdot X_i \cdot\theta_R^*(b_i)\cdot a_i^{(ZY)} Z_i b_i^{(ZY)} Y_i  c_i^{(ZY)}\cdot \theta_R^*(c_i).$$

We want to identify the two expressions using the free product structure
$$L_{V^3}^{i-1} *\grp{X_i, Z_i, Y_i, \ldots ,X_r, Z_r, Y_r}.$$
Since  $a_i,b_i,c_i\in L_{V^2}^{i-1}$, their copies in $L_{V^3}^{(XZ)}$ and $L_{V^3}^{(ZY)}$
are in $L_{V^3}^{i-1}$ \ie don't involve 
the free generators with index at least $i$.
Since $\theta_L^*(L_{V^2}^{i-1})\subset L_{V^3}^{i-1}$ and 
$\theta_R^*(L_{V^2}^{i-1})\subset L_{V^3}^{i-1}$,
the same holds for the images of $a_i,b_i,c_i$ by $\theta_L^*, \theta_R^*$. 
By identification, 
we get the desired four equations.
\end{proof}

\begin{lem}\label{lem_separation} Fix $i_0\leq r$. 
Assume that $\mu^*(T_i) = X_i Y_{i}$ for all $i < i_0$. 

With the notations of Lemma \ref{lem_positif}, we have
       $b_{i_0}\in (\LV{{i_0}-1}{X}) . (\LV{{i_0}-1}{Y})$.
\end{lem}

Before proving the lemma, recall that we write $L_V$ as the fundamental group of a graph of groups $\Lambda$ as in Theorem \ref{thm_structure_V}. 
Recall that the abelian vertex groups of $\Lambda$ are denoted by $A_j$, and that we write $A_j=\grp{c_j}\oplus N_j$.

Following \cite{Serre_arbres}, 
we define a reduced word over $\Lambda$ to be a word of the form $$\gamma_0 x_1 \gamma_1 \ldots \gamma_{n-1} x_n \gamma_n$$ 
with $\gamma_k\in\Gamma$ and $x_k\in \{T_1^{\pm 1}, \ldots, T_r^{\pm 1}\} \dunion \cup_j A_j$ satisfying the following conditions:
\begin{itemize}
    \item if $x_k\in A_j$, then $x_k\notin \grp{c_j}$;
    \item if $x_k,x_{k+1}\in A_j$ then $\gamma_k\notin \grp{c_j}$;
    \item if $x_k=T_i^{\eps}$ and $x_{k+1}=T_i^{-\eps}$ then $\gamma_{k} \neq 1$.
\end{itemize}

If two reduced words $\gamma_0 x_1 \gamma_1 \ldots \gamma_{n-1} x_n \gamma_n$ and $\gamma'_0 x'_1 \gamma'_1 \ldots \gamma'_{n'-1} x'_{n'} \gamma'_{n'}$ represent the same element, 
then by \cite[Exercice \S 5.2]{Serre_arbres} the following hold:
$(i)$ $n=n'$ and $(ii)$ for each $k$, if $x_k=T_i^{\eps}$, then $x'_k=x_k$, and if $x_k\in A_j$, then $x'_k=x_k c_j^{q}$ for some $q\in \bbZ$.

Analogous statements hold for reduced elements in $L_{V^2}$ and $L_{V^3}$ with
respect to the graphs of groups $\Lambda_{V^2}$, $\Lambda_{V^3}$.

\begin{proof}[Proof of Lemma \ref{lem_separation}]
Write $b_{i_0}\in L_{V^2}^{i_0-1}$ 
as a reduced word $b_{i_0}=\gamma_0 x_1 \gamma_1  \ldots x_n \gamma_{n} $ over $\Lambda_{V^2}$: 
here, the $\gamma_k$ are in $\Gamma$ and the $x_k$ are in $\{X_1^{\pm 1}, Y_1^{\pm 1}, \ldots, X_{i_0-1}^{\pm 1}, Y_{i_0-1}^{\pm 1}\} \dunion \cup_j A^{(XY)}_j$.

It suffices to show that there is no $Y$'s before the $X$'s in this word.
More formally, say that some element $u\in L_{V^2}$ does not involve the $X$'s (resp. the $Y$'s)
if $u\in L_{V}^{(Y)}$ (resp.\ $u\in L_{V}^{(X)}$) 
(we will use similar wordings in $L_{V^3}$).
We assume by contradiction that there are $k_1<k_2$ for which 
$x_{k_1}\notin L_{V}^{(X)}$ and $x_{k_2}\notin L_{V}^{(Y)}$
and we will show that this contradicts Equation \ref{eq_b} for $i=i_0$.

When $x_k$ lies in some abelian vertex group 
$A_{j_k}^{(XY)} =\grp{c_{j_k}}\oplus N_{j_k}^{(X)} \oplus N_{j_k}^{(Y)}$ (for some $j_k\leq l$),
we decompose $x_k$ accordingly as
$$x_k=c_{j_k}^{q_k} u_k u_k'$$
with $u_k\in N_{j_k}^{(X)}$, $u_{k'} \in N_{j_k}^{(Y)}$ not both trivial.
Then by Proposition \ref{prop_loi_abelienne},
\begin{equation}\label{eq_xk}
\theta_R^*(x_k)=c_{j_k}^{q_k} u_k^{(X)} u_k'^{(Z)}u_k'^{(Y)}\in A_j^{V^3}
\end{equation}
which does not lie in $\grp{c_{j_k}}$.

We first claim that the decomposition
$\theta_R^*(b_{i_0})=\gamma_0 \theta_R^*(x_1) \gamma_1  \ldots \theta_R^*(x_n) \gamma_n$
is a reduced word in $L_{V^3}$. 
It suffices to check that every subword of the form 
$w=\theta_R^*(x_k)\gamma_k \theta_R^*(x_{k+1})$ is reduced.

If $x_k$ and $x_{k+1}$ lie in distinct abelian vertex groups, then $w$ is reduced
by (\ref{eq_xk}).
If $x_k$ and $x_{k+1}$ lie in the same abelian vertex group $A^{(XY)}_{j_k}$, then
$\gamma_k\notin \grp{c_{j_k}}$
and $w$ is reduced.

If $x_k$ and $x_{k+1}$ are free letters; then $\gamma_k\neq 1$ if $x_{k+1}=x_k\m$.
Since by assumption, for all $i<i_0$
\begin{equation}\label{eq_free}
    \theta_R^*(X_i)=X_i, \qquad \theta_R^*(Y_i)=Z_{i}Y_i
\end{equation}

one easily deduces that $w$ is reduced.

Finally, if $x_k$ is a free letter and $x_{k+1}$ lies in some abelian vertex group or vice versa, then $w$ is clearly reduced, which proves our claim.

Since $x_{k_1}$ involves $Y$'s and $x_{k_2}$ involves $X$'s,
the reduced word for $\theta_R^*(b_0)$ 
has a letter involving $Y$'s before a letter involving $X$'s
by (\ref{eq_xk}) and (\ref{eq_free}).

When reducing $\theta_R^*(b_0)a_i^{(ZY)}$, the 
reduction will not cancel the letters involving the $X$'s
so the leftmost letter involving $Y$'s in $\theta_R^*(b_0)$ will not cancel.
Now (\ref{eq_b}), together with the results on reduced words representing the same element of \cite{Serre_arbres} referred to above, implies that $\theta_R^*(b_0)a_i^{(ZY)}\in L_{V^3}^{(XZ)}$, a contradiction.
\end{proof}

\begin{proof} [Proof of Proposition \ref{prop_changement_var}]
By induction, we may assume that the result holds for the indices $1, \ldots, i-1$. Up to a change of variables, we may thus assume that $\mu^*(T_1)=X_1Y_1, \ldots , \mu^*(T_{i-1})=X_{i-1} Y_{i-1}$.

By Lemma \ref{lem_separation}, 
there exist $b'_i,b''_i\in L_V^{i-1}$, such that 
$b_i=b_i'^{(X)} .b_i''^{(Y)}$
(where, as usual, notations such as $u^{(X)}$ denotes the copy of $u\in L_V$ in $L_V^{(X)}$).

Equation (\ref{eq_b}) gives
$$b_i'^{(X)} b''^{(Z)}_{i}= \theta_R^*\left(b_i'^{(X)} b_i''^{(Y)} \right) a^{(ZY)}_{i} =
b_i'^{(X)} \Big[\mu^*(b_i'' )\Big]^{(ZY)} a^{(ZY)}_{i}$$
hence $b''^{(Z)}_{i}=\Big[\mu^*(b_i'')\Big]^{(ZY)} a^{(ZY)}_{i}$.
Both sides of this equation live in the subgroup $L^{(ZY)}_{V^2}\subset L_{V^3}$
isomorphic to $L_{V^2}$,
so one can rewrite this equation in $L_{V^2}$ as
\begin{equation}\label{eq_bix}
b_i''^{(X)} =\mu^*(b''_{i}) a_{i}    
\end{equation}

Similarly, we using (\ref{eq_c}),
we get 
\begin{equation}\label{eq_biy}
b_i'^{(Y)}  = c_i \mu^*(b'_i).
\end{equation}

Define $T'_i = b''_i T_i b'_i$ in $L_V$, and $X'_i=b_i''^{(X)} X_i b_i'^{(X)}$,
$Y'_i=b_i''^{(Y)} X_i b_i'^{(Y)}$ in $L_{V^2}$. We have
\begin{align*}
\mu^*(T'_i)&= \mu^*(b''_i) \cdot a_i X_i  b_i'^{(X)}  b_i''^{(Y)}  Y_i  c_i \cdot \mu^*(b'_i) \\
&= \mu^*(b''_i) a_i [b_i''^{(X)} ]\m X_i'   Y_i'  [b_i'^{(Y)}]^{-1} c_i \mu^*(b'_i) \\
&= X_i' Y_i' 
\end{align*}
using (\ref{eq_bix}) and (\ref{eq_biy}).
\end{proof}

\bibliographystyle{alpha}
\bibliography{bibliography}

@book{FaMa_primer,
	Author = {Farb, Benson and Margalit, Dan},
	Isbn = {978-0-691-14794-9},
	Mrclass = {57M50 (20F36 20F65 57M07 57N05)},
	Mrnumber = {2850125 (2012h:57032)},
	Mrreviewer = {Stephen P. Humphries},
	Pages = {xiv+472},
	Publisher = {Princeton University Press, Princeton, NJ},
	Series = {Princeton Mathematical Series},
	Title = {A primer on mapping class groups},
	Volume = {49},
	Year = {2012}}

@article{Paris_beginners,
 author = {Paris, Luis},
 title = {Mapping class groups of non-orientable surfaces for beginners},
 fjournal = {Winter Braids Lecture Notes},
 journal = {Winter Braids Lect. Notes},
 issn = {2426-0312},
 volume = {1},
 pages = {ex},
 year = {2014},
 language = {English},
 doi = {10.5802/wbln.4},
 zbMATH = {7113746},
 Zbl = {1422.57049}
}

@article{Wu,
 author = {Wu, Yingqing},
 title = {Canonical reducing curves of surface homeomorphism},
 fjournal = {Acta Mathematica Sinica. New Series},
 journal = {Acta Math. Sin., New Ser.},
 issn = {1000-9574},
 volume = {3},
 pages = {305--313},
 year = {1987},
 language = {English},
 doi = {10.1007/BF02559911},
 zbMATH = {4062159},
 Zbl = {0651.57008}
}

@Article{	  bmr_discriminating,
  author	= {Baumslag, Gilbert and Myasnikov, Alexei and Remeslennikov,
		  Vladimir},
  journal	= {Geometriae Dedicata},
  title		= {Discriminating {Completions} of {Hyperbolic} {Groups}},
  year		= {2002},
  issn		= {1572-9168},
  month		= jul,
  number	= {1},
  pages		= {115--143},
  volume	= {92},
  doi		= {10.1023/A:1019687202544},
  language	= {en},
  url		= {https://doi.org/10.1023/A:1019687202544},
  urldate	= {2025-04-15}
}

@Book{		  poizat_stable,
  author	= {Poizat, Bruno},
  title		= {Stable groups. {Transl}. from the {French} by {Moses}
		  {Gabriel} {Klein}},
  fseries	= {Mathematical Surveys and Monographs},
  series	= {Math. Surv. Monogr.},
  issn		= {0076-5376},
  volume	= {87},
  isbn		= {0-8218-2685-9},
  year		= {2001},
  publisher	= {Providence, RI: American Mathematical Society (AMS)},
  language	= {English},
  zbmath	= {1614342},
  zbl		= {0969.03047}
}

@Article{	  bmr_algebraici,
  author	= {Baumslag, Gilbert and Myasnikov, Alexei and Remeslennikov,
		  Vladimir},
  coden		= {JALGA4},
  fjournal	= {Journal of Algebra},
  issn		= {0021-8693},
  journal	= {J. Algebra},
  mrclass	= {14A99 (20E99)},
  mrnumber	= {2000j:14003},
  mrreviewer	= {Dennis Spellman},
  number	= {1},
  pages		= {16--79},
  title		= {Algebraic geometry over groups. {I}. {A}lgebraic sets and
		  ideal theory},
  volume	= {219},
  year		= {1999}
}

@Article{	  besz_endomorphisms,
  author	= {Belegradek, Igor and Szczepa{\'n}ski, Andrzej},
  doi		= {10.1142/S0218196708004305},
  fjournal	= {International Journal of Algebra and Computation},
  issn		= {0218-1967},
  journal	= {Internat. J. Algebra Comput.},
  mrclass	= {20F67},
  mrnumber	= {MR2394723 (2009a:20069)},
  mrreviewer	= {Fran{\c{c}}ois Dahmani},
  note		= {With an appendix by Oleg V. Belegradek},
  number	= {1},
  pages		= {97--110},
  title		= {Endomorphisms of relatively hyperbolic groups},
  url		= {http://dx.doi.org/10.1142/S0218196708004305},
  volume	= {18},
  year		= {2008}
}

@Article{	  byronsklinos_fields,
  author	= {Dente Byron, Ayala and Sklinos, Rizos},
  journal	= {Transactions of the American Mathematical Society, Series
		  B},
  title		= {Fields definable in the free group},
  year		= {2019},
  issn		= {2330-0000},
  number	= {10},
  pages		= {297--345},
  volume	= {6},
  doi		= {10.1090/btran/41},
  language	= {English},
  url		= {https://www.ams.org/btran/2019-06-10/S2330-0000-2019-00041-9/},
  urldate	= {2025-04-08}
}

@Article{	  bow_relhyp,
  author	= {Bowditch, Brian H.},
  doi		= {10.1142/S0218196712500166},
  fjournal	= {International Journal of Algebra and Computation},
  issn		= {0218-1967},
  journal	= {Internat. J. Algebra Comput.},
  mrclass	= {20F67 (20F65)},
  mrnumber	= {2922380},
  number	= {3},
  pages		= {1250016, 66},
  title		= {Relatively hyperbolic groups},
  url		= {http://dx.doi.org/10.1142/S0218196712500166},
  volume	= {22},
  year		= {2012}
}

@Article{	  cama_virtual,
  author	= {Cashen, Christopher H. and Manning, Jason F.},
  title		= {Virtual geometricity is rare},
  journal	= {LMS J. Comput. Math.},
  fjournal	= {LMS Journal of Computation and Mathematics},
  volume	= {18},
  year		= {2015},
  number	= {1},
  pages		= {444--455},
  issn		= {1461-1570},
  mrclass	= {20E05 (20P05 57M05)},
  mrnumber	= {3367520},
  mrreviewer	= {Enric Ventura Capell},
  doi		= {10.1112/S1461157015000108},
  url		= {https://doi.org/10.1112/S1461157015000108}
}

@Article{	  cg_compactifying,
  author	= {Champetier, Christophe and Guirardel, Vincent},
  coden		= {ISJMAP},
  fjournal	= {Israel Journal of Mathematics},
  issn		= {0021-2172},
  journal	= {Israel J. Math.},
  mrclass	= {20E06 (03C45 20E05 20E26 57M07)},
  mrnumber	= {MR2151593},
  mrreviewer	= {Yves de Cornulier},
  pages		= {1--75},
  title		= {Limit groups as limits of free groups},
  volume	= {146},
  year		= {2005}
}

@Article{	  frma_stretching,
  author	= {S. Francaviglia and A. Martino},
  title		= {Stretching factors, metrics and train tracks for free
		  products},
  journal	= {Illinois J. Math.},
  year		= {2015},
  volume	= {59},
  number	= {4},
  pages		= {859-899}
}

@Article{	  gui_actions,
  author	= {Guirardel, Vincent},
  coden		= {AIFUA7},
  fjournal	= {Universit\'e de Grenoble. Annales de l'Institut Fourier},
  issn		= {0373-0956},
  journal	= {Ann. Inst. Fourier (Grenoble)},
  mrclass	= {20E08 (20E06 20F65)},
  mrnumber	= {MR2401220},
  number	= {1},
  pages		= {159--211},
  title		= {Actions of finitely generated groups on {$\Bbb R$}-trees},
  volume	= {58},
  year		= {2008}
}

@Article{	  gh_boundaries,
  author	= {Guirardel, Vincent and Horbez, Camille},
  title		= {Boundaries of relative factor graphs and subgroup
		  classification for automorphisms of free products},
  fjournal	= {Geometry \& Topology},
  journal	= {Geom. Topol.},
  issn		= {1465-3060},
  volume	= {26},
  number	= {1},
  pages		= {71--126},
  year		= {2022},
  language	= {English},
  doi		= {10.2140/gt.2022.26.71},
  zbmath	= {7525898}
}

@Article{	  gl_jsj,
  author	= {Guirardel, Vincent and Levitt, Gilbert},
  title		= {J{SJ} decompositions of groups},
  journal	= {Ast\'{e}risque},
  fjournal	= {Ast\'{e}risque},
  number	= {395},
  year		= {2017},
  pages		= {vii+165},
  issn		= {0303-1179},
  isbn		= {978-2-85629-870-1},
  mrclass	= {20E08 (20E06 20E34 20F65 20F67 57M07)},
  mrnumber	= {3758992}
}

@Article{	  gl4,
  author	= {Guirardel, Vincent and Levitt, Gilbert},
  doi		= {10.2140/gt.2011.15.977},
  fjournal	= {Geometry \& Topology},
  issn		= {1465-3060},
  journal	= {Geom. Topol.},
  mrclass	= {20E08 (20F65)},
  mrnumber	= {2821568 (2012k:20052)},
  mrreviewer	= {Fran{\c{c}}ois Dahmani},
  number	= {2},
  pages		= {977--1012},
  title		= {Trees of cylinders and canonical splittings},
  url		= {http://dx.doi.org/10.2140/gt.2011.15.977},
  volume	= {15},
  year		= {2011}
}

@Article{	  hamo_subgroup,
  author	= {M. Handel and L. Mosher},
  title		= {Subgroup decomposition in $\mathrm{Out}({F}_n)$},
  journal	= {Mem. Amer. Math. Soc.},
  year		= {2020}
}

@Book{		  hartshorne,
  author	= {Hartshorne, Robin},
  title		= {Algebraic geometry. {Corr}. 3rd printing},
  fseries	= {Graduate Texts in Mathematics},
  series	= {Grad. Texts Math.},
  issn		= {0072-5285},
  volume	= {52},
  year		= {1983},
  publisher	= {Springer, Cham},
  language	= {English},
  zbmath	= {3842033},
  zbl		= {0531.14001}
}

@Article{	  malcev_rings_and_groups,
  author	= {Mal'tsev, A. I.},
  title		= {On a correspondence between rings and groups},
  fjournal	= {Translations. Series 2. American Mathematical Society},
  journal	= {Transl., Ser. 2, Am. Math. Soc.},
  issn		= {0065-9290},
  volume	= {45},
  pages		= {221--231},
  year		= {1965},
  language	= {English},
  doi		= {10.1090/trans2/045/14},
  zbmath	= {3239513},
  zbl		= {0148.24601}
}

@Article{	  sela_diophantine7,
  author	= {Sela, Z.},
  doi		= {10.1112/plms/pdn052},
  fjournal	= {Proceedings of the London Mathematical Society. Third
		  Series},
  issn		= {0024-6115},
  journal	= {Proc. Lond. Math. Soc. (3)},
  mrclass	= {20F67 (03C10 20F70)},
  mrnumber	= {2520356 (2010k:20068)},
  mrreviewer	= {Daniel P. Groves},
  number	= {1},
  pages		= {217--273},
  title		= {Diophantine geometry over groups. {VII}. {T}he elementary
		  theory of a hyperbolic group},
  url		= {http://dx.doi.org/10.1112/plms/pdn052},
  volume	= {99},
  year		= {2009}
}

@Article{	  sela_structure,
  author	= {Z. Sela},
  coden		= {GFANFB},
  fjournal	= {Geometric and Functional Analysis},
  issn		= {1016-443X},
  journal	= {Geom. Funct. Anal.},
  number	= {3},
  pages		= {561--593},
  title		= {Structure and rigidity in ({G}romov) hyperbolic groups and
		  discrete groups in rank $1$ {L}ie groups. {I}{I}},
  volume	= {7},
  year		= {1997}
}

@Book{		  serre_arbres,
  address	= {Paris},
  author	= {Serre, Jean-Pierre},
  mrclass	= {20H10 (22E40)},
  mrnumber	= {57 \#16426},
  mrreviewer	= {E. Vinberg},
  note		= {R\'edig\'e avec la collaboration de Hyman Bass,
		  Ast\'erisque, No. 46},
  pages		= {189 pp. (1 plate)},
  publisher	= {Soci\'et\'e Math\'ematique de France},
  title		= {Arbres, amalgames, ${\rm {S}{L}}\sb{2}$},
  year		= {1977}
}

@Article{	  sklinos_fields,
  author	= {Sklinos, Rizos},
  journal	= {Proceedings of the London Mathematical Society},
  title		= {Fields interpretable in the free group},
  year		= {2024},
  issn		= {0024-6115, 1460-244X},
  month		= dec,
  number	= {6},
  pages		= {e70009},
  volume	= {129},
  doi		= {10.1112/plms.70009},
  language	= {en},
  url		= {https://londmathsoc.onlinelibrary.wiley.com/doi/10.1112/plms.70009},
  urldate	= {2025-04-08}
}

@article{Sela_diophantine1,
        Author = {Sela, Zlil},
        Fjournal = {Publications Math\'ematiques. Institut de Hautes \'Etudes Scientifiques},
        Issn = {0073-8301},
        Journal = {Publ. Math. Inst. Hautes \'Etudes Sci.},
        Mrclass = {20F65 (20E05 20E18 20E26)},
        Mrnumber = {2002h:20061},
        Mrreviewer = {Jos{\'e} Burillo},
        Pages = {31--105},
        Title = {Diophantine geometry over groups. {I}. {M}akanin-{R}azborov diagrams},
        Volume = {93},
        Year = {2001}
        }

@article{KhMy_irreducible2,
        Author = {Kharlampovich, O. and Myasnikov, A.},
        Coden = {JALGA4},
        Fjournal = {Journal of Algebra},
        Issn = {0021-8693},
        Journal = {J. Algebra},
        Mrclass = {20E10 (20E26)},
        Mrnumber = {2000b:20032b},
        Number = {2},
        Pages = {517--570},
        Title = {Irreducible affine varieties over a free group. {II}. {S}ystems in triangular quasi-quadratic form and description of residually free groups},
        Volume = {200},
        Year = {1998}}

@Article{GrWi_structure,
  author   = {Groves, Daniel and Wilton, Henry},
  journal  = {Israel Journal of Mathematics},
  title    = {The structure of limit groups over hyperbolic groups},
  year     = {2018},
  issn     = {1565-8511},
  month    = jun,
  number   = {1},
  pages    = {119--176},
  volume   = {226},
  doi      = {10.1007/s11856-018-1692-2},
  language = {en},
  url      = {https://doi.org/10.1007/s11856-018-1692-2},
  urldate  = {2025-06-25},
}

@article{Sela_diophantine6,
        Author = {Sela, Z.},
        Coden = {GFANFB},
        Fjournal = {Geometric and Functional Analysis},
        Issn = {1016-443X},
        Journal = {Geom. Funct. Anal.},
        Mrclass = {20F65 (03C07 03C10 03C65 20E05 20F10)},
        Mrnumber = {MR2238945},
        Number = {3},
        Pages = {707--730},
        Title = {Diophantine geometry over groups. {VI}. {T}he elementary theory of a free group},
        Volume = {16},
        Year = {2006}}

@misc{sela_diophantine9,
       Author = {Sela, Z.},
	title = {Diophantine {Geometry} over {Groups} {IX}: {Envelopes} and {Imaginaries}},
	url = {http://arxiv.org/abs/0909.0774},
	doi = {10.48550/arXiv.0909.0774},
	urldate = {2025-06-25},
	publisher = {arXiv},
	collaborator = {Sela, Zlil},
	month = nov,
	year = {2024},
	note = {arXiv:0909.0774 [math]},
}

@Article{	  snopcetanushevski_asymptotic,
  author	= {Snopce, Ilir and Tanushevski, Slobodan},
  journal	= {International Mathematics Research Notices},
  title		= {Asymptotic {Density} of {Test} {Elements} in {Free}
		  {Groups} and {Surface} {Groups}},
  year		= {2017},
  issn		= {1073-7928},
  month		= sep,
  number	= {18},
  pages		= {5577--5590},
  volume	= {2017},
  doi		= {10.1093/imrn/rnw175},
  url		= {https://doi.org/10.1093/imrn/rnw175},
  urldate	= {2025-02-05}
}

@Unpublished{	  brion_luminy,
  author	= {Michel Brion},
  title		= {Introduction to algebraic actions},
  note		= {\url{http://www-fourier.univ-grenoble-alpes.fr/~mbrion/notes_luminy.pdf}},
  optkey	= {},
  optmonth	= {},
  year		= {2009},
  optannote	= {}
}

@article{Stukow_Dehn,
 author = {Stukow, Micha{\l}},
 title = {Dehn twists on nonorientable surfaces},
 fjournal = {Fundamenta Mathematicae},
 journal = {Fundam. Math.},
 issn = {0016-2736},
 volume = {189},
 number = {2},
 pages = {117--147},
 year = {2006},
 language = {English},
 doi = {10.4064/fm189-2-3},
 zbMATH = {5012578},
 Zbl = {1101.57008}
}

\begin{flushleft}
Vincent Guirardel\\
Univ Rennes, CNRS, IRMAR - UMR 6625, F-35000 Rennes, France.\\
\emph{e-mail:}\texttt{vincent.guirardel@univ-rennes1.fr}\\[8mm]

Chloé Perin\\
Einstein institute of mathematics, Hebrew University of Jerusalem, Israel.\\
\emph{e-mail:}\texttt{perin@math.huji.ac.il}\\[8mm]
\end{flushleft}

\end{document}